\documentclass[11pt,english]{article}
\usepackage[T1]{fontenc}
\usepackage{mathtools}
\usepackage{enumitem}
\usepackage{amsmath}
\usepackage{amsthm}
\usepackage[symbol]{footmisc}
\usepackage{amssymb}
\usepackage{graphicx}
\usepackage[a4paper]{geometry}
\makeatletter
\numberwithin{equation}{section}
\numberwithin{figure}{section}
\theoremstyle{plain}
\newtheorem{thm}{\protect\theoremname}
\theoremstyle{definition}
\newtheorem{problem}[thm]{\protect\problemname}
\theoremstyle{remark}
\newtheorem{rem}[thm]{\protect\remarkname}
\theoremstyle{plain}
\newtheorem{cor}[thm]{\protect\corollaryname}
\theoremstyle{plain}
\newtheorem{fact}[thm]{\protect\factname}
\theoremstyle{definition}
\newtheorem{defn}[thm]{\protect\definitionname}
\theoremstyle{plain}
\newtheorem{lem}[thm]{\protect\lemmaname}
\theoremstyle{plain}
\newtheorem{prop}[thm]{\protect\propositionname}
\theoremstyle{remark}
\newtheorem*{claim*}{\protect\claimname}
\theoremstyle{remark}
\newtheorem*{summary*}{\protect\summaryname}
\newlist{casenv}{enumerate}{4}
\setlist[casenv]{leftmargin=*,align=left,widest={iiii}}
\setlist[casenv,1]{label={{\itshape\ \casename} \arabic*.},ref=\arabic*}
\setlist[casenv,2]{label={{\itshape\ \casename} \roman*.},ref=\roman*}
\setlist[casenv,3]{label={{\itshape\ \casename\ \alph*.}},ref=\alph*}
\setlist[casenv,4]{label={{\itshape\ \casename} \arabic*.},ref=\arabic*}
\theoremstyle{remark}
\newtheorem*{rem*}{\protect\remarkname}
\theoremstyle{remark}
\newtheorem{claim}[thm]{\protect\claimname}

\date{}

\theoremstyle{definition}

\usepackage{todonotes}
\usepackage[
            bookmarks=true,
            bookmarksopen=true,
            pagebackref=true,
            hyperindex=true,
            colorlinks=true,
            linkcolor=blue,
            citecolor=blue,
            filecolor=blue,
            urlcolor=blue,
			unicode,
            ]{hyperref}
\usepackage[nameinlink]{cleveref}

\crefname{problem}{problem}{problems}
\Crefname{problem}{Problem}{Problems}
\Crefformat{problem}{#2Problem #1#3}
\crefformat{problem}{#2problem #1#3}

\crefname{fact}{fact}{facts}
\Crefname{fact}{Fact}{Facts}
\Crefformat{fact}{#2Fact #1#3}
\crefformat{fact}{#2fact #1#3}

\crefname{proposition}{proposition}{propositions}
\Crefname{proposition}{Proposition}{Propositions}
\Crefformat{proposition}{#2Proposition #1#3}
\crefformat{proposition}{#2proposition #1#3}

\crefname{construction}{construction}{constructions}
\Crefname{construction}{Construction}{Constructions}
\Crefformat{construction}{#2Construction #1#3}
\crefformat{construction}{#2construction #1#3}

\crefname{claim}{claim}{claims}
\Crefname{claim}{Claim}{Claims}
\Crefformat{claim}{#2Claim #1#3}
\crefformat{claim}{#2claim #1#3}

\newcommand{\Searrow}{\Downarrow}

\makeatother

\usepackage{babel}
\providecommand{\casename}{Case}
\providecommand{\claimname}{Claim}
\providecommand{\corollaryname}{Corollary}
\providecommand{\definitionname}{Definition}
\providecommand{\factname}{Fact}
\providecommand{\lemmaname}{Lemma}
\providecommand{\problemname}{Problem}
\providecommand{\propositionname}{Proposition}
\providecommand{\remarkname}{Remark}
\providecommand{\summaryname}{Summary}
\providecommand{\theoremname}{Theorem}

\begin{document}
\title{On graphs, homology bases, and triangulated homology spheres}
\author{Karim Adiprasito\footnote{Sorbonne Université and Université Paris Cité, CNRS, IMJ-PRG, F-75005 Paris, France}, Marc Lackenby\footnote{University of Oxford}, Juan Souto\footnote{CNRS, IRMAR - UMR 6625, Université de Rennes, Campus de Beaulieu, Rennes, 35042, France}, and Geva Yashfe\footnote{Hebrew University of Jerusalem and The University of Chicago}}
\maketitle
\begin{abstract}
We describe a construction that takes as input a graph and a basis
for its first homology, and returns a triangulation of a $3$-dimensional
homology sphere. This makes precise an idea of M. Gromov and A. Nabutovski
described in \cite{gromov_spaces_and_questions}. The immediate application,
essentially described by Gromov, is to translate problems about asymptotics
of homology sphere triangulations to asymptotic counting problems
for constant-degree graphs with ``short'' homology bases. We construct
families of $3$-sphere triangulations with dual graphs that are expanders,
answering a relaxation of a question asked by G. Kalai in \cite{Kalai_f_vector_theory,Kalai_skeletons_and_paths}.
Our results also imply that if the number of $d$-dimensional triangulated
homology spheres with $n$ facets is superexponential in $n$ for
some $d$ then the same holds for $d=3$.
\end{abstract}

\section{Introduction}

\subsection{Motivation and main results}

\label[section]{sec:intro} This paper is about constructing triangulations
of the $3$-sphere and of $3$-dimensional homology spheres. This
is interesting both from an enumerative and from a more structural
point of view. In \cite[p. 33]{gromov_spaces_and_questions}, Mikhail
Gromov posed the following problem:
\begin{problem}
Fix a smooth $d$-manifold $X$ and let $t\left(X,N\right)$ be the
number of combinatorial isomorphism types of triangulations of $X$
with $N$ $d$-simplices. Does $t$ grow at most exponentially in
$N$? \label[problem]{triangulation_problem}
\end{problem}

\begin{rem}
It is known and not difficult to prove that if $d\ge2$ and $X$ is
a PL $d$-manifold then $t(X,N)$ is bounded below by an exponential
function of $N$ for all $d\ge2$, and bounded above by a function
of the form $c^{N\log N}$. This problem has a different character
from the one in which we count triangulations by the number of vertices,
for which there are faster-growing lower bounds (see \cite{Nevo_Santos_Wilson})
\end{rem}

With Alex Nabutovski, Gromov found a combinatorial analogue of \Cref{triangulation_problem}:
\begin{problem}
Evaluate the number $t_{L}\left(N\right)$ of connected $3$-regular
graphs $G$ with $N$ edges, such that cycles of length $\le L$ normally
generate $\pi_{1}\left(G\right)$ (a variation of the problem asks
only that such cycles generate $H_{1}\left(G\right)$). Is $t_{L}\left(N\right)$
at most exponential in $N$? \label[problem]{graph_problem}
\end{problem}

In \cite[p. 33]{gromov_spaces_and_questions}, Gromov states that
the essential part of \Cref{triangulation_problem} is to count triangulations
of manifolds with fixed $\pi_{1}$ or fixed $H_{1}$, and that this
is essentially equivalent to \Cref{graph_problem}. Our first main
theorem formalizes part of this statement.
\begin{thm}
\label[theorem]{thm:triangulations_vs_short_graphs}There exists $L_{0}\in\mathbb{N}$
such that for all $L\ge L_{0}$ the following holds. Let $\mathbb{F}$
be a field and $d\ge3$ an integer. Let $s_{d,\mathbb{F}}\left(N\right)$
be the number of isomorphism types of triangulated combinatorial $d$-manifolds
$M$ with $H_{1}\left(M;\mathbb{F}\right)=0$ and with at most $N$
$d$-simplices. Let $t_{L,\mathbb{F}}\left(N\right)$ be the number
of connected $3$-regular graphs $G$ with at most $N$ vertices such
that cycles of length $\le L$ generate $H_{1}\left(G;\mathbb{F}\right)$.
Then there exists $c>0$ such that
\[
s_{d,\mathbb{F}}(\frac{1}{c}N)\le t_{L,\mathbb{F}}(N)\le s_{d,\mathbb{F}}(cN)
\]
for all $N\in\mathbb{N}$.
\end{thm}

This is deduced from a general result that constructs triangulated
manifolds from bounded-degree graphs $G$ with a given set of cycles.
We can use the theorem to compare the behavior of the functions $s_{d,\mathbb{F}}$
in different dimensions $d$:
\begin{cor}
If $s_{3,\mathbb{F}}(n)=O(\exp(c\cdot n))$ for some constant $c>0$
then $s_{d,\mathbb{F}}(n)=O(\exp(c'\cdot n))$ for some $c'>0$ and
for all $d\ge4$. In particular, if $s_{d,\mathbb{F}}$ is superexponential
for some $d$ then so is $s_{3,\mathbb{F}}$. Further, if $s_{d,\mathbb{F}}(N)=\Omega(\exp(c\cdot n\log n))$
for some $c>0$ then $s_{3,\mathbb{F}}=\Omega(\exp(c'\cdot n\log n))$
for some $c'>0$.
\end{cor}

Our second main theorem is that, for each $d \ge 3$, there is an infinite family of $d$-sphere
triangulations for which the dual graphs form an expander family.
\begin{thm}
\label[theorem]{thm:expanding_S3s}For each $d \ge 3$ there is an infinite family of
combinatorially distinct triangulated $d$-spheres such that their
dual graphs form a family of $(d+1)$-regular expanders. Moreover, there is a uniform bound (depending only on the dimension $d$) on the number of simplices incident to each vertex of these triangulations.
\end{thm}

The next corollary is deduced by taking connected sums.
\begin{cor}\label[corollary]{cor:expanding_manifolds}
    Let $M$ be a triangulable $d$-manifold. Then there is an infinite family of combinatorially distinct triangulations of $M$ in which the vertex degrees are uniformly bounded (with the bound depending on $M$) and for which the dual graphs form a family of $(d+1)$-regular expanders. If $M$ is PL, there exists such a family in which the triangulations are combinatorial.
\end{cor}

An analogous result in the Riemannian setting was first proved in
unpublished work of the second and third authors: there exist metrics
on $S^{3}$ with arbitrarily large volume and Cheeger constant uniformly
bounded from $0$. The version here answers a topological version
of a question about convex polytopes asked by Gil Kalai in \cite{Kalai_f_vector_theory,Kalai_skeletons_and_paths}.
Examples of simplicial $4$-polytopes with dual graphs that are close
to expanders (but which are not quite expanders) were described by
Lauri Loiskekoski and G\"unter Ziegler in \cite{simple_polytopes_without_small_separators2}.

\subsection{Paper outline and organization}

In \Cref{sec:cycle_bases} we define \emph{short classes} of graphs
and prove basic results about them. These are essentially bounded-degree
classes of graphs with ``short'' cycle bases over a given field; in
a sense they are the main object studied in the paper. In \Cref{sec:handle_structures}
we describe a topological construction that takes as input a $2$-complex
and outputs a triangulated manifold with boundary. \Cref{sec:boundary}
discusses the relations between the topology of the boundary of such
a manifold and the topology of the input $2$-complex. In \Cref{sec:encoding_facet_colorings}
we explain how to injectively encode facet-colored triangulations
of a manifold in slightly larger (uncolored) triangulations of the
same manifold: this makes our constructions injective. In \Cref{sec:graphs_and_manifolds}
we put these tools together and prove our main results. A main part
of the construction of expanding triangulations of $S^{3}$ is deferred
to \Cref{sec:telescopes}.

\subsection*{Acknowledgments}
	K.A. is supported by the
		Centre National de Recherche Scientifique, and the Horizon Europe ERC Grant number:
		101045750 / Project acronym: HodgeGeoComb.

        M.L. was partially supported by the Engineering and Physical Sciences Research Council (grant number EP/Y004256/1).

        G.Y. was partially supported by the Horizon Europe ERC Grant number:
		101045750 / Project acronym: HodgeGeoComb. He also thanks Zlil Sela for support and helpful conversations.

        For the purpose of open access, the authors have applied a CC BY public copyright licence to any author accepted manuscript arising from this submission.
\section{Notation and preliminaries}

A triangulation $\Delta$ of a $d$-manifold $M$ is a simplicial
complex with realization homeomorphic to $M$. The dual graph of such
a triangulation $\Delta$ is the graph with vertices the set of $d$-simplices,
and edges connecting each pair of $d$-simplices that meet in a $(d-1)$-simplex;
it is always $(d+1)$-regular if $M$ has no boundary.

In a simplicial complex $\Delta$, $\Delta^{(i)}$ denotes the set
of $i$-dimensional simplices and $\Delta^{\le i}$ denotes the $i$-skeleton.
We denote the barycentric subdivision of $\Delta$ by $b(\Delta)$.

The \emph{length} $\ell(\gamma)$ of a cycle $\gamma$ in a graph
is its number of edges.

The $d$-disk is denoted $D^{d}$, and $I=[0,1]$. All topology is
done in the PL category unless stated otherwise. We follow \cite{RS_PLTop}
for the theory of regular neighborhoods. Stellar operations are defined
in \cite{Lickorish_simplicial_moves}.

\subsection{Handle structures on manifolds}

\label[section]{sec:handle_decomposition_prelims}

We recall and summarize some basic PL handle theory (details can be
found in either of \cite{Hudson_PLTop,RS_PLTop}). For the following
discussion, $I=[0,1]$, and the center point of a cube $I^{d}$ is
$\left(\frac{1}{2},\ldots\frac{1}{2}\right)$. A $d$-dimensional \emph{$k$-handle} (or \emph{handle
of index $k$}) is a space $H\overset{\text{homeo.}}{\cong}I^{k}\times I^{d-k}$
which is attached to a $d$-manifold $W$ along an embedding $f:\left(\partial I^{k}\right)\times I^{d-k}\rightarrow\partial W$.
The map $f$ is then the \emph{attaching map} of $H$ to $W\cup_{f}H$
. The \emph{core} of $H$ is $\mathrm{core}(H)=I^{k}\times \left(\frac{1}{2},\ldots\frac{1}{2}\right)$. The
manifold $W\cup_{f}H$ is thus determined by $W$ together with the
attaching map $f$. The image $\mathrm{im}(f)$ is (by definition)
a regular neighborhood of $f\left(\partial\mathrm{core}(H)\right)$
in $\partial W$: the pair $\left(\mathrm{im}(f),f(\partial\mathrm{core}(H)\right)$
is homeomorphic to $\left(\partial I^{k}\times I^{d-k},\partial I^{k}\times\left(\frac{1}{2},\ldots\frac{1}{2}\right)\right)$,
so the first space collapses onto the second (which is contained in
the first space's interior). 

A PL embedding $f:\partial I^{k} \simeq \partial \mathrm{core}(H)\rightarrow\partial W$
need not extend to an attaching map of a $k$-handle $H$ to $W$. Identifying
the image of such an $f$ with $\partial I^{k}$ and denoting a regular
neighborhood of $\partial I^{k}$ within $\partial W$ by $E$, (any
two such neighborhoods are equal up to an ambient isotopy fixing $\partial I^{k}$,)
it is clear that an embedding $f$ extends to an attaching map if
and only if the pair $\left(E,\partial I^{k}\right)$ is homeomorphic
to $\left(\partial I^{k}\times I^{d-k},\partial I^{k}\times\left(\frac{1}{2},\ldots\frac{1}{2}\right)\right)$.
This condition always holds if $k=0$ or $k=1$: if $k=0$ this is
trivial, while if $k=1$ it follows from Whitehead's theorem that
a regular neighborhood of any point in a manifold without boundary
is a disk that contains the point in its interior. For the construction
below we need the case $k=2$.
\begin{fact}
\label[fact]{fact:2_core_extension}If $k=2$, $f:\partial I^{2}\rightarrow\partial W$
extends to an attaching map of a $2$-handle if and only if a regular
neighborhood $E$ of the image is orientable. 
\end{fact}

This is well known.
\begin{proof}
[Proof sketch] We prove that a regular neighborhood $N$ of an embedded
$S^{1}$ within a PL $d$-manifold $M$ is either homeomorphic to
$\partial I^{2}\times I^{d-1}$ or nonorientable. First, choose a
$\left(d-1\right)$-disk $h$ intersecting the $S^{1}$ transversely
at a point $p$. Cut $M$ along $h$ to obtain a manifold $M'$ with
a new boundary sphere $S^{d-1}$ (consisting of two copies of $h$
glued along their boundary). The embedded $S^{1}\subset M$ is cut
into a segment, which is properly embedded in $M'$; the regular neighborhood
$N$ is cut into a regular neighborhood $N'$ of this segment. By
Whitehead's theorem, $N'$ is a $d$-disk $D$. Re-gluing $N'$ to
itself across the two copies of $h$ recovers $N$ from $N'$, and
topologically it has the effect of identifying two disjoint $\left(d-1\right)$-disks
in $\partial N'$. We denote these disks by $\left(h\cap N\right)^{+}$
and $\left(h\cap N\right)^{-}$. The triple 
\[
\left(D,\left(h\cap N\right)^{+},\left(h\cap N\right)^{-}\right)\quad\text{is homeomorphic to}\quad\left(I^{d},\left\{ 1\right\} \times I^{d-1},\left\{ 0\right\} \times I^{d-1}\right),
\]
and the gluing $N$ is therefore homeomorphic to a quotient space
of $I^{d}$ obtained by identifying $\left\{ 1\right\} \times I^{d-1}$
with $\left\{ 0\right\} \times I^{d-1}$ via some homeomorphism.
The isotopy class of this homeomorphism determines the homeomorphism
type of the quotient space, and hence of $N$. There are only two
isotopy classes of self-homeomorphisms of a disk $I^{d-1}$, only
one of which results in an orientable quotient manifold.
\end{proof}

\subsection{Quasi-isometries}

We recall some definitions and basic results.
\begin{defn}
Let $L,c>0$. An $(L,c)$ quasi-isometry between metric spaces $X$
and $Y$ is a function $f:X\rightarrow Y$ (not necessarily continuous)
such that:
\begin{enumerate}
\item For each $y\in Y$ there exists $x\in X$ such that $d\left(f(x),y\right)\le c$.
\item If $x,x'\in X$ then $\frac{1}{L}d\left(f(x),f(x')\right)-c\le d(x,x')\le Ld(f(x),f(x'))+c$.
\end{enumerate}
The spaces $X,Y$ are called quasi-isometric if there exists a quasi-isometry
$X\rightarrow Y$.

\end{defn}

\begin{rem}
\begin{enumerate}
\item If the constants don't matter much, we can also discuss $c$ quasi-isometries,
by which we mean $(c,c)$ quasi-isometries in the sense above.
\item Given an $(L,c)$ quasi-isometry $f:X\rightarrow Y$, one can construct a quasi-isometry $g:Y\rightarrow X$ as follows: for each $y \in Y$, choose $x\in X$ such that $d(f(x),y) \le c$, and set $g(y) = x$. This is a near-inverse to $f$ in the sense that
that $d(y,f\circ g(y))<c$
for all $y\in Y$ and $d(x,g\circ f(x))<c'$ for all $x\in X$, where $c'>0$ is some constant depending only on $(L,c)$.
\item A composition of $\left(L,c\right)$ quasi-isometries is again a quasi-isometry,
with constants depending only on $\left(L,c\right)$ (but generally
larger).
\item For quasi-isometry related purposes it is convenient to work with
vertex sets of graphs rather than the graphs themselves (where the
vertex sets are metrized by distance, measured by the lengths of paths).
The difference is inessential: if the vertex sets of two graphs are
$\left(L,c\right)$ quasi-isometric, and we geometrize the graphs'
realizations (as simplicial complexes, including edges) by setting
all edge lengths to be $1$, these new spaces are also $\left(L',c'\right)$
quasi-isometric where $\left(L',c'\right)$ depend only on $\left(L,c\right)$.
\end{enumerate}
The following definition and lemma are useful for producing quasi-isometries
of graphs.
\end{rem}

\begin{defn}
A metric space $X$ is $\left(L,c\right)$ \emph{quasi-geodesic} if
for any $x,x'\in X$ (with $x=x'$ allowed) there are $n+1$ points
$x_{0}=x,x_{1},\ldots,x_{n}=x'\in X$ satisfying $d(x_{i},x_{i+1})\le c$
for all $i<n$ as well as $n\le Ld(x,x')+1.$ (The $+1$ here is necessary
in case $d(x,x')<\frac{1}{L}$).
\end{defn}

Note that a connected graph is always $\left(1,1\right)$ quasi-geodesic. 
\begin{lem}
\label[lemma]{lem:quasi_isometric_relations}Let $X$ and $Y$ be
$\left(L,c\right)$ quasi-geodesic metric spaces. Suppose there exists
a relation $R\subset X\times Y$ and an $M>0$ such that:
\begin{enumerate}
\item For each $x\in X$ the set $\left\{ y\in Y\mid xRy\right\} $ is nonempty,
and similarly for each $y\in Y$ the set $\left\{ x\in X\mid xRy\right\} $
is nonempty.
\item If $x,x'\in X$ (not necessarily distinct) satisfy $d(x,x')\le c$
and $y,y'\in Y$ satisfy $xRy$ and $x'Ry'$ then $d(y,y')\le M$.
\item If $y,y'\in Y$ (not necessarily distinct) satisfy $d(y,y')\le c$
and $x,x'\in X$ satisfy $xRy$ and $x'Ry'$ then $d(x,x')\le M$.
\end{enumerate}
Then $X$ is quasi-isometric to $Y$, with quasi-isometry constants
depending only on $L,M$.

\end{lem}

\begin{proof}
We construct a quasi-isometry $f:X\rightarrow Y$ as follows. For
each $x\in X$ choose $f(x)$ to be an arbitrary element of $\left\{ y\in Y\mid xRy\right\} $.
If $x,x'\in X$ then there exist $x_{0}=x,x_{1},\ldots,x_{n}=x'\in X$
such that $d(x_{i},x_{i+1})\le c$ for all $i<n$, and in addition
$n\le Ld(x,x')+1$. By (2) we have 
\[
d(f(x),f(x'))\le d(f(x_{0}),f(x_{1}))+d(f(x_{1}),f(x_{2}))+\ldots+d(f(x_{n-1}),f(x_{n}))
\]
\[
\le Mn\le MLd(x,x')+M.
\]
Similarly, for $y=f(x)$ and $y'=f(x')$ there exist $y_{0}=y,y_{1},\ldots,y_{n}=y'\in Y$
such that $d(y_{i},y_{i+1})\le c$ for all $i<n$, and in addition
$n\le Ld(y,y')+1$. For each $0<i<n$ choose $x_{i}$ such that $x_{i}Ry_{i}$.
By (3) we have
\[
d(x,x')\le d(x_{0},x_{1})+\ldots+d(x_{n-1},x_{n})\le Mn\le MLd(f(x),f(x'))+M.
\]

Finally, note that if $y\in Y$ then there exists $x\in X$ with $xRy$,
and $d(f(x),y)\le M$ by (2), using $xRy$ and $xRf(x)$.
\end{proof}
\begin{fact}
\label[fact]{fact:expanders_via_quasi_isometry}Let $\left\{ G_{n}\right\} _{n\in\mathbb{N}}$
and $\left\{ H_{n}\right\} _{n\in\mathbb{N}}$ be families of graphs.
Assume that for each n there is an $\left(L,c\right)$ quasi-isometry
$f_{n}:H_{n}\rightarrow G_{n}$, and that there is a uniform bound
$k$ on the vertex degrees of the graphs in both families. If $\left\{ G_{n}\right\} _{n\in\mathbb{N}}$
is an expander family then so is $\left\{ H_{n}\right\} _{n\in\mathbb{N}}$.
\end{fact}

This is known: see for instance \cite[Lemma 5.8 and Remark 5.12]{Arzhantseva_origami}. It seems difficult to find a proof in the literature, so we include one for completeness.

\begin{proof}
Assume otherwise for a contradiction. Then for any $\varepsilon>0$
we may find some $n\in\mathbb{N}$ and a subset $A_{n}\subset V\left(H_{n}\right)$
containing at most half of the vertices of $H_{n}$ such that 
\[
\frac{\left|E\left(A_{n},A_{n}^{c}\right)\right|}{\left|A_{n}\right|}<\varepsilon.
\]
In particular, this implies $\left|V\left(H_{n}\right)\right|\ge\frac{1}{\varepsilon}$
and hence we may also choose such an $n$ such that $H_{n}$ has at
least $m$ vertices, for any given $m$; and since the vertex degrees
are bounded, we may also demand that $H_{n}$ have diameter larger
than any given integer. We now suppress $n$ from the notation, so
that for any given $\varepsilon>0$ we have $f:H\rightarrow G$ a
quasi-isometry and $A\subseteq V(H)$ as above, with $H$ as large
as desired and $G$ an expander (with Cheeger constant bounded below
by some fixed $\alpha>0$.) 

We use big-O notation as follows: for expressions $X,Y$ defined in
the proof (typically subsets of vertex or edge sets of $G$ and $H$,)
$X=O\left(Y\right)$ means that there exists a real $r>0$, depending
only on the parameters $L$, $c$, $k$, and $\alpha$, (but not on
$\varepsilon$,) such that $X\le r\cdot Y$ holds whenever $\varepsilon$
is small enough. Similarly, $Y=\Omega\left(X\right)$ means $X=O\left(Y\right)$
(or $Y\ge r\cdot X$ for some $r$ as above.) We also denote $X=\Theta\left(Y\right)$
if both $X=O\left(Y\right)$ and $Y=O\left(X\right)$. Each of $O$
and $\Omega$ is transitive, and $\Theta$ is an equivalence relation.

Note that $\left|V\left(H\right)\right|=\Omega\left(\left|V\left(G\right)\right|\right)$,
because each vertex of $G$ is at distance at most $c$ from a vertex
of $f\left(V\left(H\right)\right)$, so that $\left|V\left(G\right)\right|\le k^{c}\left|V\left(H\right)\right|$.
Also, if $S\subseteq V\left(H\right)$ then $\left|f\left(S\right)\right|=\Omega\left(\left|S\right|\right)$.
This holds because $f$ maps at most $k^{Lc+L+1}$ vertices of $H$
to a vertex of $G$: if $v,v'\in V\left(H\right)$ have $d\left(v,v'\right)\ge Lc+L$
then $d\left(f\left(v\right),f\left(v'\right)\right)\ge\frac{1}{L}d\left(v,v'\right)-c\ge1$.
Note that for any $d\in\mathbb{N}$ there are at most $\sum_{0\le i\le d}k^{i}\le k^{d+1}$
vertices at distance at most $d$ from any $v\in V\left(H\right)$.
In particular we have $\left|V\left(H\right)\right|=O\left(\left|V\left(G\right)\right|\right)$
and hence also $\left|V\left(H\right)\right|=\Theta\left(V\left(G\right)\right)$.

Let $N_{A}$ be a $2c+1$-neighborhood of $f\left(A\right)$. We now
show that $\left|N_{A}^{c}\right|=\Omega\left(\left|V\left(G\right)\right|\right)$.
Denote $s=\lceil\left(3c+1\right)L\rceil$. Then there are at most
$k^{s}\left|E\left(A,A^{c}\right)\right|$ vertices in $H$ which
are not in $A$ but have distance at most $s$ from $A$, because
each such vertex has a path from a nearest vertex in $A$; such a
path begins with an edge $e\in E\left(A,A^{c}\right)$, and continues
along one of at most $k^{s}$ paths of length at most $s-1$ from
the terminal vertex of $e$ in $A^{c}$. In particular, if $\varepsilon$
is small enough (so that $\left|A\right|+k^{s}\left|E\left(A,A^{c}\right)\right|\le\left(1+k^{s}\varepsilon\right)\left|A\right|\le\frac{2}{3}\left|V\left(H\right)\right|$)
then the set $B\subset V\left(H\right)$ of vertices having distance
greater than $s$ from $A$ satisfies $\left|B\right|\ge\frac{1}{3}\left|V\left(H\right)\right|$.
It follows that $\left|f\left(B\right)\right|=\Theta\left(\left|V\left(G\right)\right|\right)$.
Observe that any vertex of $f\left(A\right)$ has distance strictly
larger than $2c+1$ from any vertex of $f\left(B\right)$: for any
$a\in A$ and $b\in B$ we have 
\[
d\left(f\left(a\right),f\left(b\right)\right)\ge\frac{1}{L}d\left(a,b\right)-c>\frac{1}{L}s-c\ge2c+1,
\]
since $d\left(a,b\right)>s$. Hence $f\left(B\right)\subseteq N_{A}^{c}$,
and also 
\[
\left|N_{A}^{c}\right|\overset{*}{=}\Theta\left(\left|V\left(G\right)\right|\right).
\]

We have $\left|E\left(N_{A},N_{A}^{c}\right)\right|=\Omega\left(\left|N_{A}\right|\right)$:
if $\left|N_{A}\right|\le\frac{1}{2}\left|V\left(G\right)\right|$
this holds by $\alpha$-expansion of $G$. Otherwise, since $\left|N_{A}^{c}\right|\le\frac{1}{2}\left|V\left(G\right)\right|$,
we can apply the Cheeger inequality to $N_{A}^{c}$ and obtain
\[
\left|E\left(N_{A},N_{A}^{c}\right)\right|\ge\alpha\left|N_{A}^{c}\right|\overset{\text{by (*)}}{=}\Theta\left(\left|V\left(G\right)\right|\right)=\Omega\left(\left|N_{A}\right|\right),
\]
so the conclusion holds in either case.

Define $g:V\left(G\right)\rightarrow V\left(H\right)$ as follows:
for each $v\in V\left(G\right)$ choose a nearest vertex $f\left(w\right)\in\mathrm{im}\left(f\right)$,
and set $g\left(v\right)=w$. Then $g$ is a quasi-isometry which
is at most $k^{c+1}$-to-one, and satisfies that $d\left(v,g\circ f(v)\right)\le c'$
for all $v\in V\left(H\right)$, for some constant $c'$ depending
only on $\left(L,c\right)$. 

Let $C\subseteq N_{A}$ be the set of endpoints in $N_{A}$ of edges
in $E\left(N_{A},N_{A}^{c}\right)$. Then 
\[
\left|C\right|\ge\frac{1}{k}\left|E\left(N_{A},N_{A}^{c}\right)\right|=\Omega\left(\left|N_{A}\right|\right),
\]
and we conclude 
\[
\left|g\left(C\right)\right|=\Omega\left(\left|A\right|\right)
\]
because $N_{A}\supseteq f\left(\left|A\right|\right)=\Theta\left(\left|A\right|\right)$. 

Observe that each vertex of $g\left(C\right)$ has bounded distance
$t$ (depending only on $\left(L,c\right)$) from $A$. To see this,
note that $C\subseteq N_{A}$, so given $v\in C$ we may find $w\in A$
with $d\left(v,f\left(w\right)\right)\le2c+1$. By definition of $g$
we have $g\left(v\right)=w'$ for some $w'\in V\left(H\right)$ with
$d\left(v,f\left(w'\right)\right)\le c$, and hence $d\left(f\left(w'\right),f\left(w\right)\right)\le3c+1$.
This implies an upper bound $t$ on $d\left(w',w\right)=d\left(g\left(v\right),w\right)\ge d\left(g\left(v\right),A\right)$.
As above, there are at most $k^{t}\left|E\left(A,A^{c}\right)\right|\le\varepsilon\left|A\right|$
vertices in $H$ which are not in $A$ but have distance at most $t$
from $A$. Hence we have both $\left|g\left(C\right)\right|=O\left(\varepsilon\left|A\right|\right)$
and $\left|g\left(C\right)\right|=\Omega\left(\left|A\right|\right)$.
Taking $\varepsilon$ small enough gives a contradiction.
\end{proof}

\subsection{PL collapse}

Our notion of collapse is the same as in \cite[Ch. 3]{RS_PLTop}.
We recall the definition:
\begin{defn}
[\cite{RS_PLTop}] Suppose $X\supset Y$ are polyhedra such that $X=Y\cup D^{n}$,
where $Y\cap D^{n}$ is a face\footnote{A face $D^{n-1}$ of a disk $D^{n}$ is a subdisk of the boundary
$D^{n-1}\subset\partial D^{n}$ such that the pair $(D^{n},D^{n-1})$
is PL-homeomorphic to the pair $(I^{n-1}\times I,I^{n-1}\times\{0\})$.} $D^{n-1}$ of $D^{n}$. Then we say there is an elementary collapse
of $X$ on $Y$ and write $X\Searrow Y$ (the collapse is \emph{across}
$D^{n}$ and \emph{onto} $D^{n-1}$, \emph{from} the complementary
face $\mathrm{cl}(\partial D^{n}\setminus D^{n-1})$). We say $X$
collapses on $Y$ and write $X\searrow Y$ if there is a sequence
of elementary collapses $X=X_{0}\Searrow X_{1}\Searrow\ldots\Searrow X_{n}=Y$.
\end{defn}

Note that this notion is invariant under PL homeomorphism. It is also
useful to work with specific cellulations - this is effective (in
the sense of being computable, for example) but less flexible.
\begin{defn}
Let $X$ be polyhedron with a cell decomposition satisfying that each
cell is a subpolyhedron, the attaching maps are embeddings, and the
boundary of each cell within $X$ is a union of lower-dimensional
cells. An $(n-1)$-face $\sigma\subset X$ is \emph{free} if it is
contained in a unique $n$-face $\tau$ (so there is an elementary
collapse from $\sigma$ across $\tau$).
\end{defn}

Given a cellulation of a polyhedron, we can perform sequences of one-cell
elementary collapses if the required faces are free.

\section{Graphs, cycle bases, and 2-complexes}

\label[section]{sec:cycle_bases}

We are interested in bounded-degree graphs $G$ in which there is
a basis for $H_{1}$ consisting of ``short'' cycles. It would be
convenient to work with graphs having a cycle basis with all cycles
of uniformly bounded length, but a slight generalization is useful.
\begin{defn}
Let $R$ be a ring. A class $\mathcal{C}$ of connected finite graphs
is \emph{short} (over $R$) if there exist $d,k,L\in\mathbb{N}$ such
that the following holds: all vertex degrees of graphs in $\mathcal{C}$
are at most $d$, and for each $G\in\mathcal{C}$ there exists a collection
$\left\{ \gamma_{i}\right\} _{i\in I}$ of cycles in $G$ satisfying:
\begin{enumerate}
\item $\left\{ \gamma_{i}\right\} _{i\in I}$ spans $H_{1}\left(G;R\right)$,
\item Each edge of $G$ participates in at most $k$ of the cycles,
\item The average length $\frac{1}{\left|I\right|}\sum_{i\in I}\ell\left(\gamma_{i}\right)$
of the cycles is at most $L$.
\end{enumerate}
Such a class $\mathcal{C}$ is \emph{uniformly} short over $R$ if
(3) is strengthened to: Each $\gamma_{i}$ satisfies $\ell(\gamma_{i})\le L$.
\end{defn}

\begin{rem}
\label[remark]{rem:bounded_incidence}If the graphs in $\mathcal{C}$
have no vertices of degree $1$ or $2$ then condition (3) is redundant:
it always holds with $L=2k$. To see this, consider pairs $(e,\gamma)$
of an edge $e\in E(G)$ and a cycle $\gamma\in\left\{ \gamma_{i}\right\} _{i\in I}$
such that $e\in\gamma$. The number of pairs equals $\sum_{i\in I}\ell\left(\gamma_{i}\right)$,
and by condition (2) we find
\[
\sum_{i\in I}\ell\left(\gamma_{i}\right)\le k\cdot\left|E\left(G\right)\right|.
\]
By condition (1) we have $\left|I\right|\ge\left|E\left(G\right)\right|-\left|V\left(G\right)\right|\overset{\star}{\ge}\frac{1}{2}\left|E\left(G\right)\right|$
and dividing by $\left|I\right|$ yields the result ($\star$ holds
because $2\left|E\left(G\right)\right|\ge3\left|V\left(G\right)\right|$).

In the other direction, for uniformly short classes condition (2)
follows automatically: since the vertex degrees are bounded, there
is a uniformly bounded number of cycles of length $\le L$ passing
through any given edge.
\end{rem}

\begin{rem}
There is a universal $L_{0}$ such that all short classes can be ``encoded''
in uniformly short $3$-regular graphs with cycle length bound $L=L_{0}$:
see \Cref{lem:uniformly_short_encoding}.
\end{rem}

\begin{prop}
\label[proposition]{prop:sphere_duals_short}Let $d\ge2$. The class
$\mathcal{C}$ of all dual graphs to combinatorially-triangulated
$d$-dimensional homology spheres over a ring $R$ is a short class
over $R$.

If $R$ is a PID the same holds without restricting to combinatorial
triangulations, i.e. for the class of all dual graphs to triangulated
$d$-dimensional homology spheres over $R$.
\end{prop}

\begin{proof}
Note that graphs in $\mathcal{C}$ are all $(d+1)$-regular. Let $\Delta$
be a combinatorial triangulation of a $d$-dimensional homology sphere
over $R$. The set of $d$-simplices containing a given $\left(d-2\right)$-simplex
of $\Delta$ forms a simple cycle in the dual graph $G$. Let $\left\{ \gamma_{i}\right\} _{i\in I}$
be the collection of all such cycles, where $I=\Delta^{(d-2)}$. These
are the boundaries of $2$-faces of the dual cell complex $X$ to
$\Delta$. 

If $\Delta$ is a combinatorial triangulation then $X$ is homeomorphic
to $\Delta$, and hence has no first homology over $R$. Otherwise,
the cohomology of $\Delta$ over $R$ in each degree $i$ is isomorphic
to the homology of $X$ in degree $d-i$. Assuming that $R$ is a
PID and applying the universal coefficient theorem for cohomology
gives the exact sequence
\[
0\rightarrow\mathrm{Ext}_{R}^{1}(H_{d-2}(\Delta;R),R)\rightarrow\underbrace{H^{d-1}(\Delta;R)}_{\simeq H_{1}(X;R)}\rightarrow\underbrace{\mathrm{Hom}_{R}(H_{d-1}(\Delta;R),R)}_{=\mathrm{Hom}_{R}(0,R)=0}\rightarrow0
\]
where if $d=2$ then the Ext on the left is $\mathrm{Ext}_{R}^{1}(R,R)=0$,
and otherwise it is $\mathrm{Ext}_{R}^{1}(0,R)=0$. This implies that
the dual cell complex $X$ has no first homology over $R$. Hence,
in either case $\left\{ \gamma_{i}\right\} _{i\in I}$ (boundaries
of the $2$-cells of $X$) span $H_{1}(G;R)$ (because $G\simeq X^{\le1}$
is the $1$-skeleton of $X$). By \Cref{rem:bounded_incidence} it
suffices to show that each edge of $G$ participates in boundedly
many of the cycles $\left\{ \gamma_{i}\right\} _{i\in I}$.

A pair of edges $\left\{ u,v\right\} $ and $\left\{ v,w\right\} $
of $G$ which meet at $v$ is contained in at most one (in fact exactly
one) of the cycles $\left\{ \gamma_{i}\right\} _{i\in I}$: the corresponding
triple $\sigma_{u},\sigma_{v},\sigma_{w}$ of $d$-simplices of $\Delta$
satisfy that each of $\sigma_{u}\cap\sigma_{v}$ and $\sigma_{v}\cap\sigma_{w}$
is a $(d-1$)-simplex, and these intersections are distinct (if they
are equal we obtain a $(d-1)$-simplex contained in each of $\sigma_{u}$,
$\sigma_{v},$and $\sigma_{w}$, but this is impossible in a triangulated
$d$-manifold). Hence an edge $\{u,v\}$ is contained in at most as
many of the cycles as the number of edges containing $v$ in $G$
which are not $\{u,v\}$. Since $G$ is $(d+1)$-regular, there are
exactly $d$ such edges. 
\end{proof}
A significant part of the rest of this paper is devoted to the details
of a construction going in the other direction: from a short class
of graphs to a class of triangulated $d$-dimensional homology spheres
(where we are free to choose $d\ge3$). That this construction can
be made injective, together with the fact that it only increases the
number of faces by a constant ratio, is the main idea behind \Cref{thm:triangulations_vs_short_graphs}.
There is a natural intermediate step, from graphs to nulhomologous
$2$-complexes.
\begin{lem}
For each $m\ge3$ there exists a triangulation of the closed $2$-disk
such that:
\begin{itemize}
\item The boundary is a cycle of length $m$, and is an induced subcomplex,
\item Each vertex has degree at most $6$,
\item There are at most $m$ interior vertices, and each is connected to a boundary vertex by an edge.
\end{itemize}
\end{lem}

\begin{proof}
Begin with the zigzag triangulations of \cite[p.3]{flip_graphs_of_bounded_degree_triangulations}. For $m=8$ and $m=13$, they are:

\vspace{1ex}\includegraphics[width=0.35\textwidth]{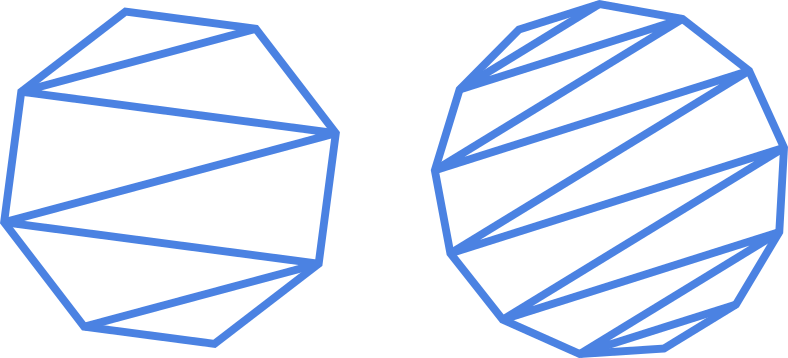}

Then take a stellar subdivision at the center of each interior edge:

\vspace{1ex}\includegraphics[width=0.35\textwidth]{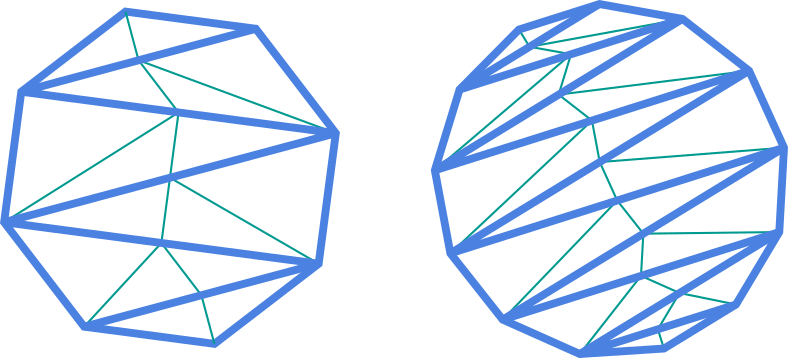}

(The order in which these subdivisions are preformed is inessential, but here we have subdivided starting from the topmost interior edge, in vertical order.)
\end{proof}
\begin{prop}
\label[proposition]{prop:graphs_to_2_complexes}Let $\mathcal{C}$
be a class of short graphs over a field $\mathbb{F}$ with constants
$d,k,L$ as above. Then there is a class $\mathcal{D}$ of connected
$2$-dimensional simplicial complexes and an integer $d'$ such that
for all $\Delta\in\mathcal{D}$ we have $H_{1}(\Delta;\mathbb{F})=H_{2}(\Delta;\mathbb{F})=0$,
and all vertices of $\Delta$ have degree at most $d'$. Further,
there exists a function $f:\mathcal{C}\rightarrow\mathcal{D}$, injective on isomorphism classes,
such that if $f(G)=\Delta$ then:
\begin{itemize}
\item $G$ is isomorphic to a subgraph of $\Delta^{\le1}$,
\item Every vertex of $\Delta$ is connected by an edge to a vertex in the
image of $G$. In particular, if $G$ has $n$ vertices then $\Delta$
has at most $d'\cdot n$ vertices.
\item If $\mathcal{C}$ is uniformly short, then $G$ is quasi-isometric
to $\Delta^{\le1}$ with quasi-isometry constants depending only on
$d,k,L$.
\end{itemize}
\end{prop}

Injectivity here is achieved by a somewhat ad-hoc process, but morally
this is similar to the coloring arguments we use elsewhere.
\begin{proof}
Given $G\in\mathcal{C}$ with cycles $\left\{ \gamma_{i}\right\} _{i\in I}$,
construct $\Delta=f(G)$ as follows. First choose $J\subset I$ such
that $\left\{ \gamma_{j}\right\} _{j\in J}$ is a basis of $H_{1}(G;\mathbb{F})$.
For each $\gamma_{j}$ glue to $G$ a triangulation of the $2$-disk
with boundary $\gamma_{j}$ as given by the previous lemma. Denote
the resulting complex by $\Delta'$.

It is straightforward that $H_{1}(\Delta';\mathbb{F})=0$. We prove
that $\Delta'$ has $H_{2}(\Delta';\mathbb{F})=0$: without loss of
generality $J=\{1,2,\ldots,t\}$ and the corresponding $2$-disks
are $D_{1},\ldots,D_{t}$. We glue the disks one at a time starting
from $D_{1}$ and prove the claim by induction on $G\cup\bigcup_{j\le i}D_{j}$,
where $G\cup\bigcup_{j\le t}D_{j}=\Delta'$. For the gluing of $D_{i}$
we have a Mayer--Vietoris sequence (with coefficients in $\mathbb{F}$,
which we omit from the notation):
\[
H_{2}\left(\gamma_{i}\right)=0\rightarrow H_{2}\left(G\cup\bigcup_{j<i}D_{j}\right)\oplus H_{2}\left(D_{i}\right)\rightarrow H_{2}\left(G\cup\bigcup_{j\le i}D_{j}\right)\rightarrow
\]
\[
H_{1}\left(\gamma_{i}\right)\rightarrow H_{1}\left(G\cup\bigcup_{j<i}D_{j}\right)\oplus H_{1}\left(D_{i}\right)\rightarrow\ldots.
\]
Since the cycles $\left\{ \gamma_{j}\right\} _{j\in J}$ form an $\mathbb{F}$-basis
of $H_{1}(G)$ and gluing $D_{1},\ldots,D_{i-1}$ has the effect on
$H_{1}$ of annihilating precisely the subspace spanned by $\gamma_{1},\ldots,\gamma_{i-1}$,
the map $H_{1}\left(\gamma_{i}\right)\rightarrow H_{1}\left(G\cup\bigcup_{j<i}D_{j}\right)$
(induced from the inclusion) is injective. It follows that the boundary
map $H_{2}\left(G\cup\bigcup_{j\le i}D_{j}\right)\rightarrow H_{1}\left(\gamma_{i}\right)$
is $0$. By induction $H_{2}\left(G\cup\bigcup_{j<i}D_{j}\right)=0$,
and hence also $H_{2}\left(G\cup\bigcup_{j<i}D_{j}\right)\oplus H_{2}\left(D_{i}\right)=0$.
So part of the sequence reads 
\[
0\rightarrow H_{2}\left(G\cup\bigcup_{j\le i}D_{j}\right)\rightarrow H_{1}\left(\gamma_{i}\right),
\]
where both maps are $0$, and therefore $H_{2}\left(G\cup\bigcup_{j\le i}D_{j}\right)=0$. 

The vertices of $\Delta'$ are the vertices of $G$, together with 
the interior vertices of the triangulated disks produced by the previous lemma. Its vertex
degrees can be bounded by $d(1+6k)$: each vertex not in $G$ has degree at most $6$ by construction.
To bound degrees of vertices in $G$, observe that each edge of $G$ is contained
in at most $k$ cycles of $\left\{ \gamma_{i}\right\} _{i\in I}$,
and hence a vertex $v$ of $G$ is contained in at most $dk$ of these
cycles. For each such cycle we have added at most $6$ edges to $v$
(here $6$ is the degree bound on the disk triangulations of the previous
lemma). In particular, each edge of $\Delta'$ participates in at
most $d(1+6k)-1$ triangles: if $e\in\Delta'$ is an edge and $v$
is one of its vertices, each triangle $\tau$ containing $e$ must
contain an additional edge $\{v,u\}$ ($u\notin e$) and this edge
identifies $\tau$ uniquely.

We construct $\Delta$ from $\Delta'$ as follows: for each edge $e=\{v,w\}$
of $G\subset\Delta'$ add $d(1+6k)$ new vertices $u_{1},\ldots,u_{d(1+4k)}$,
connect each of them to both $v$ and $w$, and add the triangles
$\{u_{i},v,w\}$ for $i=1,\ldots,d(1+6k)$. Note that each of the
new vertices is contained in precisely one triangle. Hence $G$ can
be identified within $\Delta$ as follows: it is the subgraph of the
$1$-skeleton containing all edges that participate in at least $d(1+6k)$
triangles, together with all these edges' vertices. Note further that
$\tilde{H}_{*}(\Delta;\mathbb{F})=0$, because this holds for $\Delta'$
(which is a deformation retract).

Define $\mathcal{D}$ to be the class of all complexes $\Delta$ obtained
from graphs $G\in\mathcal{C}$ as above, and define $f:\mathcal{C}\rightarrow\mathcal{D}$
by setting $f\left(G\right)$ to be the corresponding complex $\Delta$.

It remains to prove that if $\mathcal{C}$ is uniformly short then
each $G\in\mathcal{C}$ is quasi-isometric to $f(G)$, with constants
depending only on $\mathcal{C}$. Let $G\in\mathcal{C}$ and let $\Delta\supset G$
be the corresponding complex. Then the embedding $\iota:G\rightarrow\Delta^{\le1}$
is a quasi-isometry as desired: each vertex of $\Delta^{\le1}$ is
adjacent to a vertex of $G$. For any vertices $v,v'$ of $G$ we
have $d(\iota(v),\iota(v'))\le d(v,v')$. Further, if $w_{0}=\iota(v),w_{1},\ldots,w_{n}=\iota(v')$
is a path of minimal length in $\Delta^{\le1}$,
and $w_{i_0}=\iota(v),w_{i_1},\ldots,w_{i_t}=\iota(v')$ are its vertices in the image of $G$,
(where $0 =i_0 \le i_1 \le \ldots \le i_t =n$,)
then the distance in $G$ between each successive pair $v_{i_j}, v_{i_{j+1}}$ is at most $L$,
because any such pair is contained within the same cycle $\gamma_i$.
To see this, note that a connected component of the complement of $G$ within $\Delta^{\le 1}$ is either a set of the interior vertices
in a disk filling a cycle $\gamma_i$, as produced by the previous lemma, or a single vertex in $\Delta^{(0)} \setminus \Delta'^{(0)}$. In the latter case a path through the vertex is never a geodesic, because such a vertex is connected to precisely two vertices of $G$, which are already adjacent in $G$.
Hence $d(v,v')\le L\cdot n\le L\cdot d(\iota(v),\iota(v'))$.
\end{proof}

\section{Handle structures with controlled triangulations}

\label[section]{sec:handle_structures}

We describe a general ``thickening'' construction that takes $2$-dimensional
simplicial complexes $\Delta$ as input and outputs $d$-dimensional
($d\ge4$) combinatorial manifolds $M(\Delta)$ with boundary. The
results of this section can be summarized as follows:
\begin{thm}
\label[theorem]{thm:handle_construction}Fix $d,k\in\mathbb{N}$ with
$d\ge4$. Then for every $2$-dimensional simplicial complex $\Delta$
in which all vertex stars have at most $k$ faces there exists a PL
$d$-manifold $M(\Delta)$ with boundary such that:
\begin{enumerate}
\item $M(\Delta)$ has a handle structure consisting of $0$-, $1$-, and
$2$-handles in bijection with the faces of $\Delta$:
\[
M(\Delta)=\bigcup_{v\in\Delta^{(0)}}H_{0,v}\cup\bigcup_{e\in\Delta^{(1)}}H_{1,e}\cup\bigcup_{\sigma\in\Delta^{(2)}}H_{2,\sigma}.
\]
Further, $H_{i,\sigma}\cap H_{j,\tau}\neq\emptyset$ if and only if
$\sigma\cap\tau\neq\emptyset$ in $\Delta$.
\item There is a subpolyhedron $\Delta'\subset M(\Delta)$ such that $\Delta\simeq\Delta'$
are PL homeomorphic, and if $\Delta$ is finite then $M(\Delta)\searrow\Delta'$.
\item $M(\Delta)$ has a triangulation in which each handle $H_{i,\sigma}$
is a subcomplex. The total number of faces in each such subcomplex
is uniformly bounded in terms of $d,k$ (and in particular is independent
of $\Delta$).
\end{enumerate}
\end{thm}

The construction of $M(\Delta)$ from $\Delta$ involves many arbitrary
choices, but it can be made computable and explicit bounds for (3)
can be found.

The general idea is to replace each $i$-simplex in $\Delta$ by a
$d$-dimensional $i$-handle, and to glue these handles in the ``same''
manner as the faces of $\Delta$ (although this involves choices that
affect the PL homeomorphism type nontrivially). We first carry out
the construction in the PL category, then prove the wanted results
on collapsibility, and finally find suitable triangulations.

In \Cref{sec:boundary} we use part (2) to obtain conclusions on $H_{*}(\partial M)$.
Part (3) implies that the dual graph of the induced triangulation
of $\partial M$ is quasi-isometric to $\Delta$, with quasi-isometry
constants depending only on $d$ and $k$: we return to this point
in the proof of \Cref{thm:complexes_to_manifolds}.

\subsection{{\texorpdfstring{Construction of $M\left(\Delta\right)$}{Construction of M(Δ)}}}

\label[section]{sec:constructing_M_Delta}

It is useful for us to mark certain subpolyhedra on our $0$- and
$1$-handles: all gluing will take place along these subpolyhedra.
Having notation for them facilitates finding small triangulations,
because it allows us to work with fixed triangulations of the handles
(depending on the vertex degree bound $k$ but not on $\Delta$). 

\subsubsection{Constructing the $0$- and $1$-handles}

Fix an oriented $(d+1)$-disk $H_{0}$ (a ``reference $0$-handle'')
together with $k$ points $\{p_{1},\ldots,p_{k}\}$ on $\partial H_{0}$.
Take pairwise disjoint regular neighborhoods $\Sigma_{i}$ of the
points $p_{i}$ ($1\le i\le k$) in $\partial H_0$, and fix $k$ points in $\partial\Sigma_{i}$
for each $i$: call these $q_{i,1},\ldots,q_{i,k}$.

Fix a collection of pairwise disjoint paths $\{p_{t}\}_{t}$ within
$\partial H_{0}$ and disjoint from the interiors of $\Sigma_{1},\ldots,\Sigma_{k}$,
one path between each pair $q_{i_{1},j_{1}},q_{i_{2},j_{2}}$ where
$i_{1}\neq i_{2}$. Note that $\partial H_{0}$ is a $d$-sphere with
$d\ge3$, so such paths exist.

Now fix a $d$-disk $\Sigma$ together with $k$ marked points $\{u_{1},\ldots,u_{k}\}$
in $\partial\Sigma$, and define $H_{1}=\Sigma\times I$. Choose an
orientation for $H_{1}$ (our ``reference $1$-handle'').

\subsubsection{Gluing $0$- and $1$-handles}

Given an input complex $\Delta$ in which all vertex stars have at
most $k$ faces, let $(V,E)$ be the $1$-skeleton $\Delta^{\le1}$.
Take a collection of copies of $H_{0}$ indexed by $V$ and a collection
of copies of $H_{1}$ indexed by $E$: denote these by $\left\{ H_{0,v}\right\} _{v\in V}$
and $\left\{ H_{1,e}\right\} _{e\in E}$ respectively.

For each $H_{1,e}$ we want to glue each of the two ``ends'' $\Sigma\times0$
and $\Sigma\times1$ to some $\Sigma_{i}$ ($1\le i\le k$) in the
appropriate $H_{0,v}$. These should be chosen such that no two $1$-handles
are glued to the same $\Sigma_{i}\subset H_{0,v}$ for any $v$. An
arbitrary choice works (and is always possible): first, for each edge
$e=\{v,w\}$ choose an arbitrary orientation $(v,w)$. Then for each
vertex $v$ order the adjacent edges as $e_{1},\ldots,e_{s}$ (where
$s\le k$). For each $1\le i\le s$, if $v$ is the initial vertex
of the oriented edge $e_{i}$ (i.e. $e_{i}=(v,w)$ for some $w$)
we glue $\Sigma\times0\subset H_{1,e}$ to $\Sigma_{i}$; otherwise
$v$ is the terminal vertex of $e_{i}$, and we glue $\Sigma\times1\subset H_{1,e}$
to $\Sigma_{i}$. We arrange that $H_{1,e}$ and $H_{0,v}$ induce
opposite orientations on the identified disk $\Sigma\times0\simeq\Sigma_{i}$
(or $\Sigma\times1\simeq\Sigma_{i}$). This ensures that the gluing
is orientable. The gluing maps are chosen such that $\{u_{1},\ldots,u_{k}\}\subset\Sigma$
is glued to $\{q_{i,1}\ldots,q_{i,k}\}\subset\Sigma_{i}$, and we
relabel the points in $\Sigma_{i}\subset H_{0,v}$ so that the image
of each $u_{j}$ is $q_{i,j}$ (it is possible to arrange directly
for $u_{j}$ to be glued to $q_{i,j}$ without any relabeling, because
the group of PL orientation-preserving homeomorphisms of the sphere
of each dimension $\ge2$ is $k$-transitive for all $k$. But we
will not do this in the triangulations we construct). 

We denote the resulting manifold with boundary by $M_{1}$.

\subsubsection{Gluing $2$-handles}

For each triangle $\sigma=\{u,v,w\}\in\Delta^{(2)}$ we construct
a closed path $\gamma_{\sigma}$ in $\partial M_{1}$ traversing the
handles corresponding to $u,\{u,v\},v,\{v,w\},w,\{w,u\}$ in cyclic
order. The part of $\gamma_{\sigma}$ within $\partial H_{0,v}$ is
one of the paths $p_{t}$ constructed above (in the construction of
subpolyhedra of $H_{0}$) and similarly for $H_{0,w}$ and $H_{0,u}$.
For each edge $e$ of the triangle, the part of $\gamma_{\sigma}$
within $\partial H_{1,e}$ is of the form $u_{j}\times I$ for a certain
$u_{j}\in\partial\Sigma$. These choices need to be made consistently,
so that $(u_{j},0)\in\partial(\Sigma\times0)$ and $(u_{j},1)\in\partial(\Sigma\times1)$
are the endpoints of the appropriate paths $p_{t}$ in the boundaries
of the corresponding $0$-handles, and also so that the same point
$u_{j}$ does not get chosen twice within the same $1$-handle by
different triangles. Like in the case of the $1$-handles, this can
always be done: first, for each edge $e$ order the triangles $\sigma$
containing $e$ as $\sigma_{1},\ldots,\sigma_{s}$ (where $s\le k$
by assumption). Then assign each $\sigma_{i}$ the interval $u_{i}\times I$
within $\partial H_{1,e}$ (and hence also within $\partial M_{1}$).
At this point, for each triangle $\sigma$ of $\Delta$ we have chosen
the subpaths of $\gamma_{\sigma}$ within the boundaries of the $1$-handles
corresponding to its edges. If $v$ is a vertex of $\sigma$, two
of these subpaths terminate at some $q_{i_{1},j_{1}}\in\Sigma_{i_{1}}$
and $q_{i_{2},j_{2}}\in\Sigma_{i_{2}}$ within $\partial H_{0,v}$
(by our gluing of $1$-handles to $0$-handles). Exactly one of the
paths in $\{p_{t}\}_{t}$ connects these two points: this path $p_{t}$
is the subpath of $\gamma_{\sigma}$ within $H_{0,v}$.

The paths $\{\gamma_{\sigma}:\sigma\in\Delta^{(2)}\}$ are pairwise
disjoint, and since $\partial M_{1}$ is orientable (as $M_{1}$ is)
\Cref{fact:2_core_extension} shows that their regular neighborhoods
are PL-homeomorphic to $S^{1}\times D^{d-1}$. Choose pairwise disjoint
regular neighborhoods of $\{\gamma_{\sigma}:\sigma\in\Delta^{(2)}\}$
and glue a $2$-handle to each of them.

The resulting manifold is $M$.
\begin{rem}
The gluing maps of the $2$-handles involve choices that may change
the homeomorphism type. This may be easiest to see when $d=4$: in
this case the effect on $\partial M$ of gluing an additional $2$-handle
is a Dehn surgery. The choice of gluing map determines the slope,
for which an arbitrary integer can be chosen.
\end{rem}

\subsection{{\texorpdfstring{The subpolyhedron $\Delta^{\prime}$ of $M\left(\Delta\right)$}{The subpolyhedron Δ' of M(Δ)}}}

In order to identify $\Delta'$ within $M(\Delta)$ it is useful to
have notation for certain additional subpolyhedra of our handles.

Identify the reference $0$-handle $H_{0}$ with the cone over its
boundary $c_{0}*\partial H_{0}$, and similarly identify the $(d-1)$-disk
$\Sigma$ with the cone $c_{1}*\partial\Sigma$, so that our reference
$1$-handle is $H_{1}=\Sigma\times I=\left(c_{1}*\partial\Sigma\right)\times I$.
Also choose a ``center'' point $d_{i}$ in the interior of each
$\Sigma_{i}\subset\partial H_{0}$, so that $\Sigma_{i}=d_{i}*\partial\Sigma_{i}$.
Without loss of generality we may assume that the gluing maps of $1$-handles
to $0$-handles, which glue copies of $\Sigma\times i\subset\partial H_{1}$
(for $i\in\{0,1\}$) to copies of $\Sigma_{j}\subset\partial H_{0}$
(for $j\in\{1,\ldots,k\}$) satisfy that $c_{1}\in\Sigma\times i$
maps to $d_{j}\in\Sigma_{j}$. We further make the following ``linearity''
assumption: If $X\subset\partial\Sigma$ is a subpolyhedron, and $X'$
is its image under one of the gluing maps of $\Sigma\times i\subset\partial H_{1,e}$
to $\Sigma_{j}\subset\partial H_{0,v}$, then $c_{1}*X$ is glued
to $d_{i}*X'$. This assumption is harmless: it can be omitted at
the cost of introducing separate notation for the PL-homeomorph of
$c_{1}*X$ within $\Sigma_{j}$ (but it is actually satisfied in our
triangulations of $M$).

Each $2$-handle $H_{2,\sigma}$ (for $\sigma\in\Delta^{(2)}$) in
$M(\Delta)$ is a PL $I^{2}\times I^{d-2}$, and it is attached along
a regular neighborhood of $\partial I^{2}\times\{(\frac{1}{2},\ldots,\frac{1}{2})\}$
to a regular neighborhood of a closed path $\gamma_{\sigma}$. We
may assume without loss of generality that the gluing map sends $\partial I^{2}\times\{(\frac{1}{2},\ldots,\frac{1}{2})\}$
to $\gamma_{\sigma}$.

Recall that in each $e\in\Delta^{(1)}$ the $2$-faces of $\Delta$
containing $e$ are ordered as $\sigma_{1},\ldots,\sigma_{s}$ ($s\le k$)
in the construction of $M$, and that $2$-handles are glued to $\partial H_{1,e}$
along regular neighborhoods of $u_{1}\times I,\ldots,u_{s}\times I$
respectively. We denote by $s(e)$ the number of $2$-faces containing
each $e\in\Delta^{(1)}$.

We define $\Delta'$ by giving its intersection with each handle $H_{i,\sigma}$.
The idea is that these intersections contain the core of each handle
$H_{i,\sigma}$ (a point for a $0$-handle, a segment for a $1$-handle,
or a $2$-disk for a $2$-handle); but further, whenever $\sigma$
contains a face $\tau$, the intersection of $\Delta'$ with the handle
corresponding to $\tau$ has a ``flap'' extending into it from the
handle corresponding to $\sigma$. For example, the core of a $1$-handle
$H_{1,e}$ is $c_{1}\times I$, and the core of an adjacent $0$-handle
$H_{0,v}$ is $c_{0}$. We extend $c_{1}\times I$ into $H_{0,v}$
by a segment from $c_{0}$ to the corresponding endpoint of $c_{1}\times I$.
\begin{itemize}
\item For each $2$-handle $H_{2,\sigma}=I^{2}\times I^{d-2}$ define $C(\sigma)=I^{2}\times\left(\frac{1}{2},\ldots,\frac{1}{2}\right)$
(this square is the core of the handle). We define $\Delta'\cap H_{2,\sigma}=C(\sigma)$.
\item For each $1$-handle $H_{1,e}$, the core of the handle is $C(e)=c_{1}\times I\subset H_{1,e}$.
Let $\sigma_{1},\ldots,\sigma_{s}$ be the $2$-faces containing $e$
in the order above. Define the flaps
\[
F_{\sigma_{i},e}=(c_{1}*u_{i})\times I
\]
(these are rectangles with one edge going vertically through the core
of $H_{1,e}$) and set 
\[
\Delta'\cap H_{1,e}=C(e)\cup\bigcup_{i=1}^{s(e)}F_{\sigma_{i},e}.
\]
Observe that we also have $\Delta'\cap H_{1,e}=(c_{1}*\{u_{1},\ldots,u_{s(e)}\})\times I$
(or just $c_{1}\times I$ if $s(e)=0$).
\item For each $v\in\Delta^{(0)}$, the core of $H_{0,v}$ is $C(v)=\{c_{0}\}\subset H_{0,v}$.
Given an edge $e$ containing $v$, let $i$ be the unique
index $1\le i\le k$ such that $H_{1,e}$ is glued to $\partial H_{0,v}$
along $\Sigma_{i}\subset\partial H_{0,v}$. Define
\[
F_{e,v}=c_{0}*d_{i}.
\]
Given a $2$-face $\sigma$ containing $v$, denote the two edges
of $\sigma$ that contain $v$ by $e_{1},e_{2}$. Let $1\le i_{1},i_{2}\le k$
be the indices such that $H_{1,e_{1}}$ and $H_{1,e_{2}}$ are glued
to $\partial H_{0,v}$ along $\Sigma_{i_{1}}$ and $\Sigma_{i_{2}}$
respectively. Let $j_{1}$ be the index such that $\sigma$ is the
$j_{1}$-th face containing $e_{1}$ (in the order above), and similarly
$j_{2}$ for $e_{2}$. In other words, $j_{1}$ is the index such
that $F_{\sigma,e_{1}}=(c_{1}*u_{j_{1}})\times I$, and similarly
for $j_{2}$. Among the paths $\{p_{t}\}_{t}\subset\partial H_{0}$,
let $p_{t}$ be the one running between $q_{i_{1},j_{1}}$ and $q_{i_{2},j_{2}}$.
Define
\[
F_{\sigma,v}=\left[c_{0}*p_{t}\right]\cup\left[F_{e_{1},v}*q_{i_{1},j_{1}}\right]\cup\left[F_{e_{2},v}*q_{i_{2},j_{2}}\right].
\]
Note that $c_{0}*p_{t}$ is a PL $2$-disk with boundary $p_{t}\cup\left(c_{0}*q_{i_{1},j_{1}}\right)\cup\left(c_{0}*q_{i_{2},j_{2}}\right)$.
Each of the other two subpolyhedra is a triangle linearly embedded
in $H_{0,v}$, with boundaries $\{c_{0},q_{i_{1},j_{1}},d_{i_{1}}\}$
and $\{c_{0},q_{i_{2},j_{2}},d_{i_{2}}\}$ respectively. The disk
$c_{0}*p_{t}$ meets the triangle $F_{e_{1},v}*q_{i_{1},j_{1}}$ along
the segment $c_{0}*q_{i_{1},j_{1}}$ and similarly for $e_{2}$: the
union is hence a PL $2$-disk with boundary
\[
\partial F_{\sigma,v}=\underbrace{\left(c_{0}*d_{i_{1}}\right)}_{F_{e_{1},v}}\cup\left[\left(d_{i_{1}}*q_{i_{1},j_{1}}\right)\cup p_{t}\cup\left(q_{i_{2},j_{2}}*d_{i_{2}}\right)\right]\cup\underbrace{\left(d_{i_{2}}*c_{0}\right)}_{F_{e_{2},v}}.
\]

We define 
\[
\Delta'\cap H_{0,v}=C(v)\cup\bigcup_{e\ni v}F_{e,v}\cup\bigcup_{\sigma\ni v}F_{\sigma,v}.
\]

\end{itemize}
\begin{prop}
$\Delta'\simeq\Delta$.
\end{prop}

\begin{proof}
For each face $\sigma\in\Delta$ we define $D_{\sigma}=C(\sigma)\cup\bigcup_{\emptyset\neq\tau\subsetneq\sigma}F_{\sigma,\tau}$.
Observe that if $\sigma$ is a vertex then $D_{\sigma}=c_{0}\in H_{0,v}$
is a point, and if $\sigma=\{v,w\}$ is an edge then $D_{\sigma}$
is an arc with endpoints $D_{v},D_{w}$. We prove that $D_{\sigma}$
is a $2$-disk for each $\sigma\in\Delta^{(2)}$, that $\Delta'=\bigcup_{\sigma\in\Delta}D_{\sigma}$,
and that each point of $\Delta'$ is in the interior of precisely
one of the disks $\{D_{\sigma}:\sigma\in\Delta\}$ (where the unique
point of a $0$-disk is considered ``interior''). Therefore $\Delta'$
is homeomorphic to a cell complex with cells indexed by the faces
of $\Delta$. Finally, we prove that $\Delta'\simeq\Delta$ by showing
that the cells $\{D_{\sigma}\}_{\sigma\in\Delta}$ are glued to each
other in the same way as the faces of $\Delta$ are.
\begin{claim*}
Let $\sigma=\{u,v,w\}$ be a $2$-face of $\Delta$. Then $D_{\sigma}$
is a PL $2$-disk with boundary $D_{\{u,v\}}\cup D_{\{v,w\}}\cup D_{\{w,u\}}$.
\end{claim*}
\begin{proof}
By definition,
\[
D_{\sigma}=C(\sigma)\cup F_{\sigma,\{u,v\}}\cup F_{\sigma,\{v,w\}}\cup F_{\sigma,\{w,u\}}\cup F_{\sigma,u}\cup F_{\sigma,v}\cup F_{\sigma,w}
\]
is a union of seven $2$-disks. Their interiors are disjoint, because
the interior of each is contained in the interior of a different handle
of $M$. It follows from the construction that the disks are arranged
as \vspace{1ex}

\includegraphics{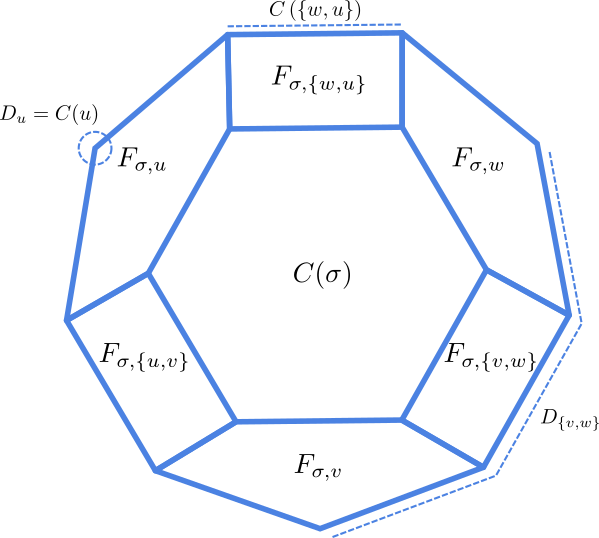} 

and hence form a disk as desired. The dashed line on the right is
parallel to the boundary arc $D_{\{v,w\}}$. The rest of the boundary
is similarly composed of $D_{\{w,u\}}\cup D_{\{u,v\}}$.
\end{proof}
\begin{claim*}
$\Delta'=\bigcup_{\sigma\in\Delta}D_{\sigma}$.
\end{claim*}
\begin{proof}
This is direct from the definition: for each face $\tau\in\Delta$
of dimension $i$, the subpolyhedron $\Delta'\cap H_{i,\tau}$ is
of the form $C(\tau)\cup\bigcup_{\sigma\supset\tau}F_{\sigma,\tau}$.
Hence
\[
\Delta'=\bigcup_{\sigma\in\Delta}C(\sigma)\cup\bigcup_{\emptyset\neq\tau\subsetneq\sigma\in\Delta}F_{\sigma,\tau}.
\]
Each $D_{\sigma}$ is of the form $D_{\sigma}=C(\sigma)\cup\bigcup_{\emptyset\neq\tau\subsetneq\sigma}F_{\sigma,\tau}$,
and hence $\Delta'=\bigcup_{\sigma\in\Delta}D_{\sigma}$.
\end{proof}
\begin{claim*}
For each $\sigma\in\Delta$ the subpolyhedron $D_{\sigma}\subset\Delta'$
is a disk of dimension $\dim\sigma$. If $\sigma,\tau\in\Delta$ are
distinct faces then the interiors of the disks $D_{\sigma}$ and $D_{\tau}$
are disjoint.
\end{claim*}
\begin{proof}
For $\sigma$ a vertex, $D_{\sigma}=C(\sigma)$ is a point. For $\sigma=\{v,w\}$
an edge, $D_{\sigma}=C(\sigma)\cup F_{\sigma,v}\cup F_{\sigma,w}$
is a union of three segments: for the point $c_{0}\in H_{0,v}$ and
the index $1\le i\le k$ such that $H_{1,\sigma}$ is glued to $H_{0,v}$
along $\Sigma_{i}\subset\partial H_{0,v}$, we have $F_{\sigma,v}=c_{0}*d_{i}$,
where $d_{i}$ is one of the two endpoints of $C(\sigma)$. Hence
$D_{\sigma}$ is a $1$-disk (an arc) with endpoints $C(v)$ and $C(w)$.
For $\sigma$ a triangle the previous claim shows that $D_{\sigma}$
is a $2$-disk. This proves the first part of the claim.

Let $\sigma,\tau\in\Delta$ be distinct. Note that $D_{\sigma}=C(\sigma)\cup\bigcup_{\emptyset\neq\rho\subset\sigma}F_{\sigma,\rho}$,
where $C(\sigma)$ is contained in $H_{\dim\sigma,\sigma}$ and each
$F_{\sigma,\rho}$ is contained in $H_{\dim\rho,\rho}$. In particular,
in order to compute $D_{\sigma}\cap D_{\tau}$ it suffices to consider
the intersection within those handles of $M$ that correspond to faces
$\rho\subseteq\sigma\cap\tau$. For $\rho\subseteq\sigma\cap\tau$
of dimension $i$, we consider the possibilities for $D_{\sigma}\cap D_{\tau}\cap H_{i,\rho}$:
\begin{itemize}
\item The case $i=2$ is impossible, because if $\sigma\cap\tau$ is a $2$-face
then $\tau=\sigma$.
\item If $i=1$ then $D_{\sigma}\cap H_{1,\rho}$ is either $C(\rho)$ (if
$\rho=\sigma$) or of the form $F_{\sigma,\rho}=(c_{1}*u_{j})\times I$
for some $1\le j\le k$ if $\rho\subsetneq\sigma$. Similarly $D_{\tau}\cap H_{1,\rho}=F_{\tau,\rho}$
is either $C(\rho)$ or $(c_{1}*u_{j'})\times I$ for some $1\le j'\le k$,
where $j\neq j'$ if both are defined. In all cases we have $F_{\sigma,\rho}\cap F_{\tau,\rho}=C(\rho)$.
\item If $i=0$ then $\rho=\{v\}$ is a vertex. Recall that $C(v)=\{c_{0}\}\subset H_{0,v}$.
If $\sigma$ is a proper superset of $\{v\}$ then $F_{\sigma,\rho}$
is a cone $c_{0}*Z$ for some $Z\subset\partial H_{0,v}$: $Z$ is
either a single point $d_{j}$ ($1\le j\le k$) if $\sigma$ is the
edge such that $H_{1,\sigma}$ is glued onto $\Sigma_{j}\subset\partial H_{0,v}$,
or (if $\sigma$ is a $2$-face) it is an arc of the general form
$\left(d_{i_{1}}*q_{i_{1},j_{1}}\right)\cup p_{t}\cup\left(q_{i_{2},j_{2}}*d_{i_{2}}\right)$
for some indices $1\le i_{1},i_{2},j_{1},j_{2}\le k$ (where $i_{1}\neq i_{2}$).
Similarly if $\tau$ is a proper superset of $\rho$ then $F_{\tau,\rho}=c_{0}*Z'$
for $Z'$ of the same general form as $Z$. If either $\sigma=\rho=\{v\}$
then $D_{\sigma}\cap D_{\tau}\cap H_{0,v}=D_{\sigma}\cap D_{\tau}=C(v)$,
which is a boundary point of $D_{\tau}$, and the interiors of the
disks are disjoint. Otherwise $\rho$ is a proper face of both $\sigma$
and $\tau$, and $F_{\sigma,\rho}\cap F_{\tau,\rho}=c_{0}*(Z\cap Z')$.
If $Z\cap Z'=\emptyset$ then $F_{\sigma,\rho}\cap F_{\tau,\rho}=\{c_{0}\}$
is a boundary point of $D_{\sigma}$ and of $D_{\tau}$ and again
we are done. Otherwise, $Z\cap Z'=\{d_{i_{1}}\}$ for some $1\le i_{1}\le k$.
In this case at least one of $\sigma$ or $\tau$ is a $2$-face (else
both are edges, but this would imply that $H_{1,\sigma}$ and $H_{1,\tau}$
have both been glued to $\Sigma_{i_{1}}\subset\partial H_{0,v}$,
but this contradicts our choices above). Assuming without loss of
generality that $\sigma$ is a $2$-face, $c_{0}*d_{i_{1}}\subset H_{0,v}$
is an arc on the boundary of $D_{\sigma}$, and hence there is no
interior point of $D_{\sigma}$ in $D_{\sigma}\cap D_{\tau}\cap H_{0,v}$. 
\end{itemize}
\end{proof}
Observe that each point of $\Delta^{\prime\le1}\coloneqq\bigcup_{\{\sigma\in\Delta:\dim\sigma\le1\}}D_{\sigma}$
is in the interior of precisely one of the disks $D_{\sigma}$: if
it is not in the interior of some arc $D_{e}$ (where $e\in\Delta^{(1)}$
is an edge) then it is the boundary point of such an arc, which is
an interior point of the corresponding $0$-disk. Since for each $2$-face
$\tau\in\Delta$ the $2$-disk $D_{\tau}$ is glued onto $\Delta^{\prime\le1}$
along the entire boundary, each point of $\Delta'$ is either a point
of $\Delta^{\prime\le1}$ or is in the interior of some disk in $\{D_{\tau}:\tau\in\Delta^{(2)}\}$.
Hence $\Delta'$ is homeomorphic to a cell complex with cells $\{D_{\sigma}:\sigma\in\Delta\}$.
(If $\Delta$ happens to be infinite, the topology induced on $\Delta'$
from $M$ is the weak topology\footnote{This follows from the fact that in our triangulation of $M$ it is
possible to realize $\Delta'$ as a subcomplex after a constant number
of barycentric subdivisions.}.)

A PL homeomorphism $\Delta\simeq\Delta'$ can now be described by
induction on skeleta, where $\Delta^{\prime\le i}\coloneqq\{D_{\sigma}:\sigma\in\Delta,\dim\sigma\le i\}$.
We have $\Delta^{\le0}\simeq\Delta^{\prime\le0}$, since both are
$0$-dimensional and their vertex sets are in bijection $v\mapsto D_{v}$.
Also $\Delta^{\le1}\simeq\Delta^{\prime\le1}$, because for each edge
$e=\{u,v\}$ we have $\partial D_{e}=D_{u}\cup D_{v}$, and hence
the map $\Delta^{\le0}\overset{\sim}{\rightarrow}\Delta^{\prime\le0}$
extends to $1$-skeleta. It also extends to $2$-skeleta, because
for each $\sigma=\{u,v,w\}\in\Delta^{(2)}$ the boundary of $D_{\sigma}$
is precisely the cycle $D_{\{u,v\}}\cup D_{\{v,w\}}\cup D_{\{w,u\}}$,
which is the image of $\partial\sigma$ under the map $\Delta^{\le1}\overset{\sim}{\rightarrow}\Delta^{\prime\le1}$.
\end{proof}

\subsection{{\texorpdfstring{Collapsing $M\left(\Delta\right)$ onto $\Delta^{\prime}$}{Collapsing M(Δ) onto Δ'}}}

We prove that for each $i$-face $\sigma\in\Delta$ it is possible
to collapse $H_{i,\sigma}$ onto the union of $H_{i,\sigma}\cap\Delta'$
and the subpolyhedron along which $H_{i,\sigma}$ is glued to lower-index
handles. We then show that these operations can be carried out within
$M$ in order of decreasing handle index, so that $M(\Delta)\searrow\Delta'$. 

Our main tool is the next lemma.
\begin{lem}
\label[lemma]{lem:handle_collapse}Let $P\subset\partial I^{d-k}$
be a subpolyhedron, and identify the $(d-k)$-cube $I^{d-k}$ with
the cone $c*\partial I^{d-k}$ over its boundary. Then 
\[
I^{d}=I^{k}\times I^{d-k}\searrow\left[I^{k}\times\left(c*P\right)\right]\cup\left[\partial I^{k}\times I^{d-k}\right].
\]
\end{lem}

\begin{proof}
Choose a triangulation $X$ of $\partial I^{d-k}$ such that $P$
is a subcomplex, so that $c*X$ is a triangulation of $I^{d-k}$.
Identify $I^{k}$ with the $k$-simplex, which we
denote $Z$ (as $\Delta$ is taken). Consider the cellulation of $I^{k}\times I^{d-k}$
with cells given by $\left\{ \sigma\times\tau \mid \sigma \in Z, \tau \in c*X \right\} $.
Denote by $\delta \in Z$ the unique $k$-dimensional face. 

Order $X\setminus P$ (i.e. the set of faces of $X$ which are not
faces of $P$) in order of decreasing dimension as $\tau_{1},\tau_{2},\ldots,\tau_{k}$.
Observe that $\delta\times\tau_{1}$ is contained in no face other
than $\delta\times\left(c*\tau_{1}\right)$, and hence we may collapse
$\delta\times\left(c*\tau_{1}\right)$ from $\delta\times\tau_{1}$
onto $\delta\times\left(c*\partial\tau_{1}\right)$. We do this for
every $\tau_{i}$ in turn for $2\le i\le k$, collapsing $\delta\times\left(c*\tau_{i}\right)$
from $\delta\times\tau_{i}$. By the end of the process we have deleted
precisely $\delta\times\left(c*X\setminus c*P\right)$, where $\delta$
is the interior of the closed disk $Z$. Hence we are left with
\[
\left[\delta\times\left(c*P\right)\right]\cup\left[\partial Z \times\left(c*X\right)\right]=\left[Z \times\left(c*P\right)\right]\cup\left[\partial Z\times\left(c*X\right)\right]
\]
(where the inclusion $\subset$ is clear, and $\supset$ follows from
the fact $Z=\delta\cup\partial Z$, so that $Z \times\left(c*P\right)$
is contained in the union on the left). 

By our identification of $I^{k}$ and $I^{d-k}$ with the appropriate
triangulations we have obtained precisely $I^{d}=I^{k}\times I^{d-k}\searrow\left[I^{k}\times\left(c*P\right)\right]\cup\left[\partial I^{k}\times I^{d-k}\right]$.
\end{proof}
We apply this lemma for each handle $H_{i,\sigma}$, where $k=i=\dim\sigma$
and $H_{i,\sigma}$ is identified with $I^{k}\times I^{d-k}$: 
\begin{itemize}
\item If $\sigma=\{v\}$ is a vertex, the lemma states that $I^{d}$ collapses
onto $c*P$ for any subpolyhedron $P\subset\partial I^{d}$. Since
$\Delta'\cap H_{0,v}$ is of the form $c_{0}*P$ for a certain subpolyhedron
$P\subset\partial H_{0,v}$, we get the desired collapse $H_{0,v}\searrow\Delta'\cap H_{0,v}$.
\item If $\sigma=e=\{u,v\}$ is an edge, $\Delta'\cap H_{1,e}=(c_{1}*\{u_{1},\ldots,u_{s(e)}\})\times I$
for $c_{1}\in\Sigma$ the chosen center point and $s(e)$ the number
of $2$-faces containing $e$. To bring this description to the form
in the lemma, identify $H_{1,e}=\Sigma\times I$ with $I\times I^{d-1}$
(where the $\Sigma$ factor maps to the $I^{d-1}$ factor) and set
$P\subset I^{d-1}$ the image of the subpolyhedron $\{u_{1},\ldots,u_{s(e)}\}$
of $\Sigma$. In terms of $H_{1,e}$, we obtain 
\[
H_{1,e}\searrow\left[\Delta'\cap H_{1,e}\right]\cup\left[\Sigma\times\{0\}\right]\cup\left[\Sigma\times\{1\}\right].
\]
Observe that this is the union of $\Delta'\cap H_{1,e}$ with the
part of $\partial H_{1,e}$ that is glued onto the $0$-handles of
$M$.
\item If $\sigma$ is a $2$-face, we apply the lemma with $P$ empty, so
that $c*P=\{c\}$. We obtain a collapse $I^{2}\times I^{d-2}\searrow\left[I^{2}\times c\right]\cup\left[\partial I^{2}\times I^{d-2}\right]$.
Identifying $c$ with the interior point $\left(\frac{1}{2},\ldots,\frac{1}{2}\right)$
of $I^{d-k}$, we get 
\[
I^{2}\times I^{d-2}\searrow\left[I^{2}\times\left(\frac{1}{2},\ldots,\frac{1}{2}\right)\right]\cup\left[\partial I^{2}\times I^{d-2}\right].
\]
Here $I^{2}\times\left(\frac{1}{2},\ldots,\frac{1}{2}\right)=C(\sigma)$,
and altogether $\left[I^{2}\times\left(\frac{1}{2},\ldots,\frac{1}{2}\right)\right]\cup\left[\partial I^{2}\times I^{d-2}\right]$
is a union of $C(\sigma)$ and a regular neighborhood of $C(\sigma)\cap\partial I^{d}$,
by the next lemma. In other words, we obtain a collapse of $H_{2,\sigma}$
onto the union of $\Delta'\cap H_{2,\sigma}$ with the part of $\partial H_{2,\sigma}$
that is glued onto the $0$- and $1$-handles of $M$.
\end{itemize}
\begin{lem}
$(\partial I^{2})\times I^{d-2}$ is a regular neighborhood of $\left[I^{2}\times\left\{ \left(\frac{1}{2},\ldots,\frac{1}{2}\right)\right\} \right]\cap\partial I^{d}$
within $\partial I^{d}$.
\end{lem}

\begin{proof}
By \cite[Cor. 3.30]{RS_PLTop} it suffices to show that $(\partial I^{2})\times I^{d-2}\subset\partial I^{d}$
is a compact manifold with boundary and that it collapses onto $\left[I^{2}\times\left\{ \left(\frac{1}{2},\ldots,\frac{1}{2}\right)\right\} \right]\cap\partial I^{d}$. 

It is clear that $(\partial I^{2})\times I^{d-2}\simeq S^{1}\times D^{d-2}$
is a compact manifold with boundary.

For the statement on collapsing, first observe that $\left[I^{2}\times\left\{ \left(\frac{1}{2},\ldots,\frac{1}{2}\right)\right\} \right]\cap\partial I^{d}=\left(\partial I^{2}\right)\times\left\{ \left(\frac{1}{2},\ldots,\frac{1}{2}\right)\right\} $.
So we need to prove that $\left(\partial I^{2}\right)\times I^{d-2}\searrow\left(\partial I^{2}\right)\times\left\{ \left(\frac{1}{2},\ldots,\frac{1}{2}\right)\right\} $.
We prove more generally that $P\times I^{n}\searrow P\times\left\{ \left(\frac{1}{2},\ldots,\frac{1}{2}\right)\right\} $
for any compact polyhedron $P$. It suffices to prove this for $n=1$
by induction: if the claim is true for some $n$ then 
\[
P\times I^{n+1}=\left(P\times I\right)\times I^{n}\searrow\left(P\times I\right)\times\left\{ \left(\frac{1}{2},\ldots,\frac{1}{2}\right)\right\} ,
\]
and up to permuting coordinates the polyhedron on the right is just
$\left(P\times\left\{ \left(\frac{1}{2},\ldots,\frac{1}{2}\right)\right\} \right)\times I$,
which we know collapses onto $P\times\left\{ \left(\frac{1}{2},\ldots,\frac{1}{2}\right)\right\} \times\left\{ \frac{1}{2}\right\} =P\times\underbrace{\left\{ \left(\frac{1}{2},\ldots,\frac{1}{2}\right)\right\} }_{\in I^{d+1}}$.

For $n=1$, fix a triangulation of $P$ and triangulate $I$ as a
union of two edges $\left[0,\frac{1}{2}\right]\cup\left[\frac{1}{2},1\right]$.
Thus $P\times I$ has a cellulation in which the cells are $\left\{ \sigma\times\tau\mid\sigma\in P,\tau\in I\right\} $.
For each $\dim\left(P\right)$-dimensional face $\sigma$ of $P$
we have that $\sigma\times\left\{ 1\right\} $ is a free face of $P\times I$
(it is contained only in $\sigma\times\left[\frac{1}{2},1\right]$)
and similarly $\sigma\times\left\{ 0\right\} $ is a free face of
$P\times I$ (contained only in $\sigma\times\left[0,\frac{1}{2}\right]$).
Collapsing $\sigma\times\left[\frac{1}{2},1\right]$ and $\sigma\times\left[0,\frac{1}{2}\right]$
along these free faces , we see 
\[
P\times I\searrow\left(P\times\left\{ \frac{1}{2}\right\} \right)\cup\left\{ \sigma\times\tau\mid\sigma\in P,\tau\in I,\dim\sigma\le\dim P-1\right\} .
\]
We continue in this way with the faces of $P$ of dimension $\dim\left(P\right)-1,\ldots,0$:
at the $i$-th step, we have collapsed $P\times I$ onto
\[
\left(P\times\left\{ \frac{1}{2}\right\} \right)\cup\left\{ \sigma\times\tau\mid\sigma\in P,\tau\in I,\dim\sigma\le\dim P-i\right\} ,
\]
and for each $\left(\dim P-i\right)$-dimensional face $\sigma$ of
$P$ we have that $\sigma\times\left\{ 1\right\} $ is contained only
in $\sigma\times\left[\frac{1}{2},1\right]$ and $\sigma\times\left\{ 0\right\} $
is contained only in $\sigma\times\left[0,\frac{1}{2}\right]$, so
we may collapse further onto
\[
\left(P\times\left\{ \frac{1}{2}\right\} \right)\cup\left\{ \sigma\times\tau\mid\sigma\in P,\tau\in I,\dim\sigma\le\dim P-(i+1)\right\} .
\]
Continuing this way we eventually collapse $P\times I$ onto $P\times\left\{ \frac{1}{2}\right\} $
as desired.
\end{proof}
\begin{prop}
If $\Delta$ is finite then $M(\Delta)\searrow\Delta'$.
\end{prop}

For the proof, it is useful to have a criterion for when a collapse
$X\searrow Y$ extends to a collapse $X\cup Z\searrow Y\cup Z$.
\begin{lem}
\label[lemma]{lem:extending_collapse}Let $X\subset W$ be polyhedra
and let $Z=\mathrm{cl}_{W}(W\setminus X)$. Let $Y\subset X$ be a
subpolyhedron such that $X\searrow Y$ by a sequence of elementary
collapses $X=X_{0}\Searrow X_{1}\Searrow\ldots\Searrow X_{t}=Y$.
For each $0\le i<t$ denote by $\sigma_{i}\subset\tau_{i}$ the disks
such that $X_{i}\Searrow X_{i+1}$ is a collapse across $\tau_{i}$
from $\sigma_{i}$. Assume that $Z$ is disjoint from the interiors
of $\sigma_{i}$ and $\tau_{i}$ for all $i$. Then $X\cup Z\searrow Y\cup Z$.
\end{lem}

\begin{proof}
We first prove that $X_{0}\cup Z\Searrow X_{1}\cup Z$. For brevity
denote $\sigma=\sigma_{0}$ and $\tau=\tau_{0}$. We have $X_{0}=X_{1}\cup\tau$
where $X_{1}\cap\tau$ is a face of $\tau$ and $\sigma$ is the complementary
face to $X_{1}\cap\tau$, meaning $\sigma=\mathrm{cl}_{\tau}\left(\partial\tau\setminus X_{1}\cap\tau\right)$.
Hence $X_{0}\cup Z=X_{1}\cup Z\cup\tau$ and 
\[
\left(X_{1}\cup Z\right)\cap\tau=\left(X_{1}\cap\tau\right)\cup\left(Z\cap\tau\right).
\]
By assumption, $Z$ is disjoint from the interiors of $\sigma$ and
$\tau$. We have $Z\cap\tau\subseteq X_{1}\cap\tau$, because $\tau\setminus X_{1}$
consists precisely of the union of the interiors of $\tau$ and $\sigma$
(so $\tau\setminus X_{1}\subseteq\tau\setminus Z$). Hence $\left(X_{1}\cup Z\right)\cap\tau=X_{1}\cap\tau$
is a face of $\tau$ complementary to $\sigma$, and we may collapse
$X_{1}\cup Z\cup\tau=X_{0}\cup Z$ onto $X_{1}\cup Z$.

Observe that $X_{1}\subset X_{1}\cup Z$ are polyhedra, and $Z=\mathrm{cl}_{X_{1}\cup Z}\left(\left[X_{1}\cup Z\right]\setminus X_{1}\right)$:
to see the equality, note that $Z$ is a closed subset of $X_{1}\cup Z$
(which is itself a closed subset of $X\cup Z$) and that $\left[X_{1}\cup Z\right]\setminus X_{1}=\left[X\cup Z\right]\setminus X$,
because the interiors of $\sigma$ and $\tau$ are in $X$. Hence,
setting $W_{1}=X_{1}\cup Z$, we may apply the claim inductively to
the shorter sequence $X_{1}\Searrow\ldots\Searrow X_{n}=Y$, with
$X_{1}\subset W_{1}$ taking the place of $X\subset W$, and the collapse
$X_{1}\searrow Y$ taking the place of $X\searrow Y$. 
\end{proof}
\begin{proof}
[Proof of the proposition]As described above, we order the handles
by decreasing index ($2$-handles first) as $H_{1},\ldots,H_{n}$
(the ordering induced on the $i$-handles for each $0\le i\le2$ is
arbitrary). Then, within $M$ we collapse each handle $H_{i}$ onto
$H_{i}\cap\Delta'\cup\left[H_{i}\cap\bigcup_{j>i}H_{j}\right]$, in
order. We use \Cref{lem:extending_collapse} to show that this is
indeed possible. 

Suppose the handles have been ordered as $H_{1},\ldots,H_{n}$. After
collapsing $H_{1},\ldots,H_{i-1}$ onto their intersections with $\Delta'$
we are left with $M_{i}\coloneqq\Delta'\cup\bigcup_{j\ge i}H_{j}$.
Since $M_{i}=H_{i}\cup\left(M_{i}\setminus H_{i}\right)$, it suffices
to show that $\mathrm{cl}_{M_{i}}\left(M_{i}\setminus H_{i}\right)$
is disjoint from the interiors of the disks removed in the collapse
$H_{i}\searrow H_{i}\cap\Delta'$. We prove this with $\mathrm{cl}_{M_{i}}\left(M_{i}\setminus H_{i}\right)$
replaced by the larger set $M_{i+1}$. Observe that $M_{i+1}$ intersects
the interior of $H_{i}$ only in $\Delta'$ (because distinct handles
of $M$ intersect only in their boundaries). Also, every elementary
collapse in \Cref{lem:handle_collapse} is performed from a face $\sigma$
in $\partial H_{i}$ across a face $\tau$ whose interior is entirely
contained in the interior of $H_{i}$. Therefore, for such an elementary
collapse it suffices to check that $M_{i+1}$ is disjoint from the
interior of $\sigma$. Also, $\Delta'$ is disjoint from the interior
of $\sigma$ (since $\Delta'\cap H_{i}$ is contained in the subpolyhedron
of $H_{i}$ obtained by deleting the interiors of $\sigma$ and of
$\tau$). Hence it suffices to check that for all $j>i$ the handle
$H_{j}$ is disjoint from the interior of $\sigma$. Note that any
two $k$-handles ($k\in\{0,1,2\}$) within $M$ have disjoint boundaries.
Therefore it suffices to verify that the collapse of $H$ described
in \Cref{lem:handle_collapse} occurs from subdisks of $\partial H$
that do not intersect any lower-index handle. But the collapse is
precisely onto the union of $H\cap\Delta'$ with the part of $\partial H$
glued onto the lower-index handles; so no subdisk in the intersection
of $\partial H$ with a lower-index handle is removed in the process.
\end{proof}

\subsection{{\texorpdfstring{Triangulating $M\left(\Delta\right)$}{Triangulating M(Δ)}}}

We construct triangulations of the handles and their subpolyhedra
as defined in \Cref{sec:constructing_M_Delta}. These triangulations
depend on $d=\dim M$ and on $k$ (the bound on vertex links of $\Delta$)
but not on $\Delta$ itself. 

It is convenient to demand that $k$ is even, by increasing it by
$1$ if necessary.

\subsubsection{Triangulating the $0$- and $1$-handles}

Choose a combinatorially triangulated $(d-1)$-sphere $S_{0}$ with
an automorphism group that has an orbit of $k$ vertices, which we
denote $d_{1},\ldots,d_{k}$. We additionally require that for each
$i$ there is an orientation-reversing automorphism that fixes $d_{i}$.
An example of such a triangulation is a join $X*Y$, where $X$ is
a triangulation of $S^{1}$ with $k$ vertices and $Y$ is the boundary
of a $(d-k-1)$-simplex.

Replace $S_{0}$ by its second barycentric subdivision: all automorphisms
are still available because the subdivision respects all symmetries.
Also, the closed stars of $d_{1},\ldots,d_{k}$ are pairwise isomorphic
and disjoint $(d-1)$-disks, which we denote $\Sigma_{1},\ldots,\Sigma_{k}$.
By assumption each $\Sigma_{i}$ has an orientation-reversing automorphism
that fixes the center point $u_{i}$. We choose such automorphisms
$\varphi_{1},\ldots,\varphi_{k}$ such that each pair $\varphi_{i}$,
$\varphi_{j}$ are conjugate via an automorphism of $S_{0}$ taking
$u_{i}$ to $u_{j}$.

Take further barycentric subdivisions until each of the disks $\Sigma_{i}$
has $k$ points in its boundary. These play the role of $q_{i,1},\ldots,q_{i,k}$
in \Cref{sec:constructing_M_Delta}. We demand that $q_{i,1},\ldots,q_{i,k}$
are a disjoint union of $\varphi_{i}$-orbits $\left\{ q_{i,1},q_{i,2}\right\} ,\ldots,\left\{ q_{i,k-1},q_{i,k}\right\} $
for each $i$, and that choices are made symmetrically in the sense
that for each $1\le i,j\le k$ there exists an automorphism $\psi_{i,j}$
of $S_{0}$ taking this set of orbits to the set of $\varphi_{j}$-orbits
$\left\{ \left\{ q_{j,1},q_{j,2}\right\} ,\ldots,\left\{ q_{j,k-1},q_{j,k}\right\} \right\} $.
(This can be done by choosing the points within $\Sigma_{i}$ and
then taking their images under an appropriate isomorphism $\Sigma_{1}\rightarrow\Sigma_{j}$
for each $j>1$). Then, take further barycentric subdivisions until
$S_{0}$ contains a system of pairwise disjoint paths $\left\{ p_{t}\right\} _{t}$
connecting $q_{i_{1},j_{1}}$ with $q_{i_{2},j_{2}}$ for all pairs
with $i_{1}\neq i_{2}$ (where $1\le i_{1},i_{2},j_{1},j_{2}\le k$),
where the paths are disjoint from the interiors of the disks $\Sigma_{1},\ldots,\Sigma_{k}$.
The cone $c_{0}*S_{0}$ is (almost) our reference handle $H_{0}$:
$S_{0}$ will be replaced with a further subdivision during the construction
of the $1$-handles.

At this point, take a copy $\Sigma$ of one of the (pairwise isomorphic)
disks $\Sigma_{i}$. Denote the copies of the points $q_{i,1},\ldots,q_{i,k}$
within $\Sigma$ by $u_{1},\ldots,u_{k}$. There is a cellulation
of $\Sigma\times I$ in which the cells are all products of the form
$\sigma\times\tau$ for $\sigma\in\Sigma$ and $\tau\in I$ (where
$I$ is triangulated by a single edge). Take a combinatorial triangulation
refining this cellulation, in a way that the induced triangulation
of $\Sigma\times\{0\}$ and $\Sigma\times\{1\}$ is precisely the
original triangulation of $\Sigma$. One way of doing this is to stellate
in the interior of each face which is not in $\Sigma\times\{0\}$
or $\Sigma\times\{1\}$, in order of decreasing dimension. In this
triangulation of $\Sigma\times I$ the subpolyhedra $u_{i}\times I$
are pairwise disjoint subcomplexes. Further barycentrically subdivide
this triangulation twice in order to obtain that the simplicial neighborhoods
of the subcomplexes $u_{i}\times I$ are regular neighborhoods. This
is our triangulation of $H_{1}$.

Finally, replace $S_{0}$ with its second barycentric subdivision,
so that the copies of (subdivisions of) $\Sigma_{1},\ldots,\Sigma_{k}$
within it are isomorphic to our triangulation of $\Sigma\times\{0\}$
and $\Sigma\times\{1\}$. Our triangulation of $H_{0}$ is $c_{0}*S_{0}$.
Note that the further subdivision ensures that simplicial neighborhoods
of the paths $\left\{ p_{t}\right\} _{t}$ are pairwise disjoint regular
neighborhoods.
\begin{prop}
$M(\Delta)$ has a triangulation in which each handle $H_{i,\sigma}$
is a subcomplex. The total number of faces in each such subcomplex
is uniformly bounded in terms of $d,k$ (and in particular is independent
of $\Delta$).
\end{prop}

\begin{proof}
It is possible to glue together $M_{1}$ from our triangulations of
$H_{0}$ and $H_{1}$, where the gluing maps are simplicial isomorphisms
gluing $\Sigma\times\{0\}$ or $\Sigma\times\{1\}$ in a copy of $H_{1}$
onto some $\Sigma_{i}$ in a copy of $H_{0}$, because there are simplicial
isomorphisms $\Sigma\times\{0\}\overset{\sim}{\rightarrow}\Sigma_{i}$
and $\Sigma\times\{1\}\overset{\sim}{\rightarrow}\Sigma_{i}$. In
fact, there are at least two such simplicial isomorphisms, with opposite
orientations - given one, a second can be obtained by composing with $\varphi_{i}$.
So the requirements on orientation in $M_{1}$ can also be met: by
construction of $q_{i,1},\ldots,q_{i,k}$ as a $\varphi_{i}$-invariant
set, both of these isomorphisms map $u_{1},\ldots,u_{k}\in\partial\left(\Sigma\times\{0\}\right)$
to $q_{i,1},\ldots,q_{i,k}\in\partial\Sigma_{i}$.

Note that the triangulation of $M_{1}$ we have obtained satisfies
the conditions of the proposition, and it remains to triangulate the
$2$-handles in some way. Each $2$-handle $H_{2,\sigma}$ is glued
onto a regular neighborhood of a path $\gamma_{\sigma}$ within $\partial M_{1}$,
which in our triangulation is a subcomplex. In fact, by construction
our triangulation of $M_{1}$ is the second barycentric subdivision
of a complex in which $\gamma_{\sigma}$ is a subcomplex; in $\partial M_{1}$,
its simplicial neighborhood is a regular neighborhood (and this regular
neighborhood of $\gamma_{\sigma}$ is disjoint from the regular neighborhood
of $\gamma_{\tau}$ for any $2$-face $\tau\neq\sigma$).

The simplicial neighborhood of $\gamma_{\sigma}$ in $\partial M_{1}$
has one of only a finite number of isomorphism types. To see this,
observe that if $\sigma=\{u,v,w\}$ then the simplicial neighborhood
of $\gamma_{\sigma}$ is a subcomplex of our triangulation of 
\[
H_{0,u}\cup H_{0,v}\cup H_{0,w}\cup H_{1,\{u,v\}}\cup H_{1,\{v,w\}}\cup H_{1,\{w,u\}},
\]
which is a union of a bounded number of handles and hence has a uniformly
bounded number of faces; there are only finitely many such isomorphism
types of subcomplexes of simplicial complexes (and their number depends
only on $k$ and $d$). Let $\Gamma_{1},\ldots,\Gamma_{N}$ be the
total number of possible isomorphism types of simplicial neighborhoods
of the paths $\gamma_{\sigma}\subset\partial M_{1}$. For each isomorphism
type $\Gamma_{i}$ fix a triangulation of a $2$-handle $I^{2}\times I^{d-2}$
satisfying that the induced triangulation of $\partial I^{2}\times I^{d-2}$
is isomorphic to $\Gamma_{i}$: this is possible by \cite{Armstrong_extending_triangulations}.
Finiteness ensures that the number of faces in the triangulation of
each $2$-handle can be bounded in terms of $d$ and $k$ alone.
\end{proof}

\section{Topology of the boundary}

\label[section]{sec:boundary}

The goal of this section is to prove the following theorem.
\begin{thm}
\label[theorem]{thm:boundary}Let $M$ be a connected manifold.
\begin{enumerate}
\item If $M$ is PL-collapsible then $\partial M$ is a sphere.
\item If $\partial M$ is connected and $R$ is a PID such that $H_{k}\left(M;R\right)=0$
for all $k>0$ then $\partial M$ is a triangulated $R$-homology
sphere.
\item If $M$ is contractible, $\dim M\ge5$, and $M$ has a handle structure
consisting only of 0-, 1-, and 2-handles, then $\partial M$ is a
homotopy sphere.
\end{enumerate}
\end{thm}

\begin{rem}
All three parts are routine or known, but we include the proof for
completeness. The assumption that $\partial M$ is connected in (2)
is necessary: consider $\mathbb{RP}^{2}\times I$ over $R=\mathbb{Q}$.
One can replace it with the assumption that $2$ is not invertible
in $R$.
\end{rem}

\begin{proof}
Denote $d=\dim M$.
\begin{enumerate}
\item If $M$ is PL-collapsible then $M$ is a PL disk by a theorem of Whitehead
(see \cite[Corollary 3.28]{RS_PLTop}). Therefore $\partial M$ is
a sphere.
\item Consider the long exact sequence
\[
\ldots\rightarrow H_{k}\left(M;R\right)\rightarrow H_{k}\left(M,\partial M;R\right)\rightarrow H_{k-1}\left(\partial M;R\right)\rightarrow H_{k-1}\left(M;R\right)\rightarrow\ldots
\]
of the pair $\left(M,\partial M\right)$. The assumption that $H_{k}\left(M;R\right)=0$
for all $k>0$ implies that $H_{k}\left(M,\partial M;R\right)\cong H_{k-1}\left(\partial M;R\right)$
for all $k>1$.

Observe that around $k=1$ the sequence reads
\[
0=H_{1}\left(M;R\right)\rightarrow H_{1}\left(M,\partial M;R\right)\rightarrow H_{0}\left(\partial M;R\right)\rightarrow H_{0}\left(M;R\right)\rightarrow0=H_{0}\left(M,\partial M;R\right)
\]
where $H_{0}\left(M;R\right)\simeq R$ and $H_{0}\left(\partial M;R\right)\simeq R$
by the assumption that $\partial M$ is connected. Hence $H_{1}\left(M,\partial M;R\right)=0$
so that the pair $\left(M,\partial M\right)$ is $R$-orientable,
$H_{d}\left(M,\partial M;R\right)\simeq R$, and Lefschetz duality
applies.

By Lefschetz duality $H_{k}\left(M,\partial M;R\right)\simeq H^{d-k}\left(M;R\right)$.
The universal coefficient theorem for cohomology states that for each
$k$ there is an exact sequence of the form 
\[
0\rightarrow\mathrm{Ext}_{R}^{1}\left(H_{k-1}\left(M;R\right),R\right)\rightarrow H^{k}\left(M;R\right)\rightarrow\mathrm{Hom}_{R}\left(H_{k}\left(M;R\right),R\right)\rightarrow0.
\]
If $k>0$ we have $\mathrm{Hom}_{R}\left(H_{k}\left(M;R\right),R\right)=\mathrm{Hom}_{R}\left(0,R\right)=0$.
Further, if $k>1$ we have 
\[
\mathrm{Ext}_{R}^{1}\left(H_{k-1}\left(M;R\right),R\right)=\mathrm{Ext}_{R}^{1}\left(0,R\right)=0,
\]
while if $k=1$ we have
\[
\mathrm{Ext}_{R}^{1}\left(H_{k-1}\left(M;R\right),R\right)=\mathrm{Ext}_{R}^{1}\left(R,R\right)=0
\]
since $R$ is free. Therefore
\[
0=H^{k}\left(M;R\right)\simeq H_{d-k}\left(M,\partial M;R\right)\simeq H_{d-k-1}\left(\partial M;R\right)\text{ for all \ensuremath{0<k<d-1}}.
\]
(For $k=d-1$ this fails because $0=H_{1}\left(M,\partial M;R\right)\not\simeq H_{0}\left(\partial M;R\right)\simeq R$,
as already observed above).

Finally, $H_{d-1}\left(\partial M;R\right)\simeq H_{d}\left(M,\partial M;R\right)\simeq R$,
where the last isomorphism holds because $\left(M,\partial M\right)$
is $R$-orientable.
\item By part (2), $\partial M$ is a $\mathbb{Z}$-homology sphere. A simply
connected homology sphere is a homotopy sphere, so it suffices to
prove $\partial M$ is simply-connected.

By eliminating the $0$-handles beyond the first from the given handle
decomposition and reordering, we may assume that $M$ is obtained
from a disk $D$ (the unique $0$-handle) by attaching a collection
of mutually disjoint $1$-handles and then a collection of $2$-handles
which are disjoint from each other (but not from the $1$-handles):
there is a sequence of the form
\[
D=M_{0},M_{1},\ldots,M_{r},M_{r+1},\ldots,M_{r+s}=M,
\]
such that $M_{i+1}$ is obtained from $M_{i}$ by adding a $1$-handle
for each $i<r$, and $M_{r+i+1}$ is obtained from $M_{r+i}$ by adding
a $2$-handle for $i<s$. The boundary $\partial M$ can be described
by a sequence of surgeries on the boundary $S^{d-1}$ of the disk
$D$, one for each $k$-handle added to $M$. More explicitly, given
an attaching map $f_{j}:S^{k-1}\times D^{d-k}\rightarrow\partial M_{j}$
of a $k$-handle to $M_{j}$, $\partial M_{j+1}$ can be described
as the result of the surgery step that removes $f_{j}\left(S^{k-1}\times\mathrm{int}\left(D^{d-k}\right)\right)$
and glues a copy of $D^{k}\times S^{d-k-1}$ onto $f_{j}\left(S^{k-1}\times S^{d-k-1}\right)$. 

The inclusion map $\partial M_{j}\hookrightarrow M_{j}$ induces an
isomorphism on fundamental groups for each $j$. Choose some $x\in\partial D$
such that $x$ is not in the image of any $f_{j}$ (and hence $x\in\partial M_{j}$
for each $j$). We prove the claim by induction on $j$: for $j=0$
it is clear. If $j+1\le r$ then $M_{j+1}$ is obtained by adding
a $1$-handle to $M_{j}$, and $\partial M_{j+1}$ is obtained by
removing two disjoint open disks from $\partial M_{j}$ and gluing
a cylinder $S^{d-2}\times I$ between their boundaries. Both steps
add a free generator to the fundamental group, so that $\pi_{1}\left(M_{r},x\right)\simeq F_{r}\simeq\pi_{1}\left(\partial M_{r},x\right)$,
and the inclusion $\partial M_{r}\hookrightarrow M_{r}$ induces such
an isomorphism. If $r<j+1\le r+s$ then a $2$-handle is added to
$M_{j}$ along a map $f_{j}:S^{1}\times D^{d-2} \rightarrow \partial M_j$, which annihilates
the conjugacy class of the restriction to $S^{1}\times\left\{ 0\right\} $
in $\pi_{1}\left(M_{j},x\right)$. Since $\dim\partial M_{j}\ge4$,
the corresponding surgery on $\partial M_{j}$ has precisely the same
effect of annihilating the conjugacy class of the corresponding element
in $\pi_{1}\left(\partial M_{j},x\right)$: see \cite[Lemma 2]{Milnor_killing_homotopy_groups}.
\end{enumerate}
\end{proof}

\section{Encoding facet colorings in manifold triangulations}

\label[section]{sec:encoding_facet_colorings}

Let $\Delta$ be a triangulated $d$-dimensional homology manifold
and $c:\Delta^{(d)}\rightarrow\left[r\right]$ a coloring of the facets
with $r$ colors. We wish to encode the pair $\left(\Delta,c\right)$
in a single triangulation of a homology manifold $\Delta^{\prime}$,
in such a way that $\Delta^{\prime}$ is not much larger than $\Delta$.
In this section we provide such a construction and prove some of its
basic properties.

Note that for many purposes (e.g. asymptotic enumeration) it would
suffice to work with colored triangulations: so long as the number
of colors $r$ is uniformly bounded, the number of pairs $\left(\Delta,c\right)$
is larger than the number of triangulated manifolds $\Delta$ with
$N$ facets by only an exponential factor.
\begin{defn}
We call a pair $\left(\Delta,c\right)$ of a $d$-dimensional homology
manifold $\Delta$ (without boundary) and a facet coloring $c:\Delta^{(d)}\rightarrow\left[r\right]$
an\emph{ $r$-colored triangulated homology $d$-manifold. }An isomorphism
of $r$-colored triangulated homology $d$-manifolds $\left(\Delta_{1},c_{1}\right)\overset{\sim}{\rightarrow}\left(\Delta_{2},c_{2}\right)$
is a simplicial isomorphism $\phi:\Delta_{1}\rightarrow\Delta_{2}$
such that for each facet $\sigma$ of $\Delta_{1}$ it holds that
$c_{2}\left(\phi(\sigma)\right)=c_{1}\left(\sigma\right)$.
\end{defn}

\begin{thm}
\label[theorem]{thm:color_encoding}For each $d,r\in\mathbb{N}$ with
$d\ge2$ there is a constant $m\in\mathbb{N}$ and a function 
\[
\mathrm{Enc}_{d,r}:\left\{ \text{\ensuremath{r}-colored triangulated homology \ensuremath{d}-manifolds}\right\} \rightarrow\left\{ \text{triangulated homology \ensuremath{d}-manifolds}\right\} 
\]
such that if $\Delta^{\prime}=\mathrm{Enc}_{d,r}\left(\Delta,c\right)$
then:
\begin{enumerate}
\item $\Delta^{\prime}$ is a subdivision of $\Delta$ in which each facet
is replaced with a combinatorial $d$-disk on at most $m$ facets.
In particular, there is a PL homeomorphism $\Delta^{\prime}\simeq\Delta$
and $\left|\Delta^{\prime(d)}\right|\le m\left|\Delta^{(d)}\right|$.
\item If each vertex of $\Delta$ is incident to at most $n$ edges, then
each vertex of $\Delta^{\prime}$ is incident to at most $mn$ edges,
\item In addition, $\mathrm{Enc}_{d,r}$ is injective on isomorphism classes.
\end{enumerate}
\end{thm}

\begin{rem}
The functions $\mathrm{Enc}_{d,r}$ we obtain are easy to compute,
and so are their left inverses.
\end{rem}

\begin{summary*}
The idea of the proof is to subdivide $\Delta$, and replace each
of its facets $\sigma$ with a PL-triangulated disk $\sigma^{\prime}$
such that the triangulation of $\sigma^{\prime}$ encodes the color
$c(\sigma)$. It is automatic that condition (1) of the theorem holds
for such a construction. Since we only need a finite ``palette''
of disk triangulations $\sigma^{\prime}$, condition (2) holds as
well. There is a mild difficulty with condition (3): it is not clear
that $\Delta$ can be reconstructed from the resulting complex. Our
strategy is to first replace $\Delta$ with its barycentric subdivision
$\mathrm{b}\left(\Delta\right)$, and the coloring $c$ by the facet
coloring $\mathrm{b}(c)$ that it induces (each facet of $\mathrm{b}\left(\Delta\right)$
is given the color of the unique facet of $\Delta$ that it refines).
We then perform a further subdivision to encode the induced facet
coloring.

The barycentric subdivision is useful because in $\mathrm{b}\left(\Delta\right)$
every vertex has high degree. Our further subdivision will introduce
only vertices of lower degree, so that we will be able to identify
the vertices of $b\left(\Delta\right)$ in the final subdivided complex
$\Delta^{\prime}$. This will ensure that $\mathrm{b}\left(\Delta\right)$
and its coloring $\mathrm{b}(c)$ can both be reconstructed from the
resulting complex. Except for certain degenerate cases, $\Delta$
can be reconstructed from $\mathrm{b}\left(\Delta\right)$ by \cite{finney_barycentric_insufficiency}.
We enrich the facet coloring $c$ with more information in order to
deal with these degenerate cases.
\end{summary*}
\begin{lem}
\label[lemma]{lem:barycentric_d_plus_2}Let $\Delta$ be a triangulated
homology $d$-manifold without boundary ($d\ge2$), and denote its
barycentric subdivision by $\mathrm{b}\left(\Delta\right)$. Then
each vertex of $\mathrm{b}\left(\Delta\right)$ is incident to at
least $d+2$ edges.
\end{lem}

\begin{proof}
A vertex $v$ of $\mathrm{b}\left(\Delta\right)$ is the barycenter
of a nonempty face $\sigma$ of $\Delta$. Denote $k=\dim\sigma$.
Each vertex $w$ neighboring $v$ in the 1-skeleton of $\mathrm{b}\left(\Delta\right)$
is either the barycenter of a face $\tau$ satisfying $\emptyset\subsetneq\tau\subsetneq\sigma$
(a total of $2^{k+1}-2$ faces) or the barycenter of a face $\tau$
satisfying $\tau\supsetneq\sigma$. The number of faces that properly
contain $\sigma$ is the number of nonempty faces in $\mathrm{lk}\left(\sigma,\Delta\right)$,
which is a homology sphere of dimension $d-k-1$. 

We split into cases depending on the value of $s=\dim\mathrm{lk}\left(\sigma,\Delta\right)=d-k-1$.
Note that $s\ge-1$.
\begin{casenv}
\item If $s=-1$ then $k=d$ and $\sigma$ contains $2^{k+1}-2=2^{d+1}-2\ge d+2$
proper nonempty faces. Thus $v$ neighbors at least $d+2$ vertices.
\item If $s=0$ then $k=d-1$, and $\sigma$ contains $2^{k+1}-2=2^{d}-2$
proper nonempty faces, and is contained in precisely $2$ faces of
$\Delta$ (since $\Delta$ is a homology manifold, and the link of
$\sigma$ is a $0$-sphere). Again $v$ neighbors at least
\[
2^{d}\ge d+2
\]
vertices in the 1-skeleton of $\mathrm{b}\left(\Delta\right)$. 
\item Suppose $s\ge1$. Any homology sphere $\Sigma$ of dimension $s\ge1$
has at least $s+2$ vertices and $s+2$ facets: it has at least one
facet $\rho$ of dimension $s$, which has $s+1$ vertices, and at
least one vertex not on this face. Each of the maximal proper faces
of $\rho$ is a ridge (or codimension-1 face) of $\Sigma$, and hence
contained in a facet of $\Sigma$ other than $\rho$. Hence there
are at least $s+1$ facets of $\Sigma$ other than $\rho$ itself.
It therefore suffices to check that
\[
\left(2^{k+1}-2\right)+2\left(d-k-1+2\right)\ge d+2.
\]
Observe that
\[
\left(2^{k+1}-2\right)+2\left(d-k+1\right)=2\left(d+2^{k}-k\right)
\]
and that $2^{k}-k\ge1$, which proves the claim.
\end{casenv}
\end{proof}
The family of triangulated disks in the next definition will form
the ``color palette'' for our constructions.
\begin{defn}
Let $d\in\mathbb{N}$. We define $\mathcal{S}_{d}$ to be the family
of simplicial complexes that can be obtained from the $d$-simplex
by iterated stellar subdivision at facets. In other words, a complex
$\Delta$ is in $\mathcal{S}_{d}$ if there exists a finite sequence
of complexes $\Delta_{0},\Delta_{1},\ldots,\Delta_{n}$ such that
$\Delta_{0}$ is the $d$-simplex, $\Delta_{n}=\Delta$, and $\Delta_{i+1}$
is obtained from $\Delta_{i}$ by stellating a facet.
\end{defn}

\begin{rem*}
Note that any $\Delta\in\mathcal{S}_{d}$ is a triangulated PL disk,
and its boundary is isomorphic to the boundary of the simplex. Note
also that $\Delta_{1}$ in any sequence as above is uniquely defined:
it is the cone over the boundary of the $d$-simplex.
\end{rem*}
\begin{lem}
\label[lemma]{lem:coloring_d_plus_2}Let $\left(\Delta,c\right)$
be an $r$-colored triangulated homology $d$-manifold ($d\ge2)$
in which all vertices have degree at least $d+2$. Choose any $r$
non-isomorphic complexes $\Sigma_{1},\ldots,\Sigma_{r}\in\mathcal{S}_{d}$.
Construct a new complex $\Delta^{\prime}$ as follows: delete the
interior of each facet $\sigma$ of $\Delta$, and glue $\Sigma_{c(\sigma)}$
to the deleted complex via an (arbitrary) simplicial isomorphism $\partial\Sigma_{c(\sigma)}\rightarrow\partial\sigma$.
Then the isomorphism class $\left(\Delta,c\right)$ can be reconstructed
from $\Delta^{\prime}$. 
\end{lem}

In particular, this construction is injective on the set of all isomorphism
classes of $r$-colored triangulated homology $d$-manifolds ($d\ge2$)
with all vertex degrees at least $d+2$.
\begin{claim}
If $\Delta_{0}=\Delta,\Delta_{1},\ldots,\Delta_{k}$ is a sequence
in which each $\Delta_{i+1}$ is constructed from $\Delta_{i}$ by
stellar subdivision at a facet, and in $\Delta$ all vertices have
degree at least $d+2$, then the process of greedily performing reverse
stellations at facets while possible always terminates with $\Delta$.
\end{claim}

\begin{proof}
Note that reverse stellar subdivision at a facet is precisely the
removal of the open star of a vertex of degree $d+1$ (whose link
must be a simplicial sphere with $d+1$ vertices, hence is the boundary
of a simplex) and the replacement of this open star by the interior
of a single $d$-simplex. On the $1$-skeleton this process has precisely
the effect of deleting a vertex and its edges. Hence we can reverse-stellate
and delete a given vertex if and only if its degree in the $1$-skeleton
is $d+1$. Note that this means that a vertex of $\Delta$ can never
be deleted in the greedy process: if $v\in\Delta$ is the first vertex
deleted this way, all its neighbors in $\Delta$ must have been present
before the deletion, at which point its degree in the $1$-skeleton
was at least $d+2$ by assumption.

It is convenient to realize $\Delta$ in $\mathbb{R}^{k}$ (for some
large $k$, e.g. $2d+1$) in order to use the following geometric
notion of subdivision: a subdivision of $\Delta$ is a simplicial
complex in $\mathbb{R}^{k}$ with the same support as $\Delta$, and
such that every face is contained in a unique face of $\Delta$. 

Call a vertex of a subdivision of $\Delta$ \emph{young} if it is
not a vertex of $\Delta$. The neighbors of a young vertex are all
contained in one facet of $\Delta$, and any reverse stellation which
deletes a given (necessarily young) vertex $v$ preserves this property
(since it preserves the neighbors' positions). It follows that a reverse
stellation at a facet of a subdivision of $\Delta$ always results
in a subdivision of $\Delta$. The young vertices of $\Delta_{k}$
are ordered by age (each vertex is introduced at a different $\Delta_{i}$,
with the youngest being the last added,) and if $\overline{\Delta}$
is the result of any sequence of reverse facet stellations from $\Delta_{k}$,
a reverse facet stellation deleting a vertex $v$ of $\overline{\Delta}$
may be performed whenever $v$ is younger than all its neighbors in
$\overline{\Delta}^{\le1}$: if $v$ has no younger neighbor then
its degree is at most $d+1$, but the degree of a vertex in a $d$-dimensional
homology manifold is always at least $d+1$, so $\deg_{\overline{\Delta}}(v)=d+1$.
Hence the process can never get stuck except at $\Delta$ (in which
we know no reverse stellation is possible): if $\overline{\Delta}\neq\Delta$,
we may always reverse-stellate further by deleting the youngest vertex
left in $\overline{\Delta}$. 
\end{proof}
\begin{proof}
[Proof of the lemma] The claim above shows that iteratively performing
all possible reverse stellations at facets recovers $\Delta$ from
$\Delta^{\prime}$. We can embed $\Delta$ in some $\mathbb{R}^{k}$
and re-perform the stellations to obtain $\Delta^{\prime}$ as a geometric
subdivision of $\Delta$; the combinatorial isomorphism type of the
simplicial disk refining a given facet $\sigma$ of $\Delta$ contains
precisely the information of the color $c(\sigma)$.
\end{proof}
\begin{proof}
[Proof of \Cref{thm:color_encoding}] The main theorem of \cite{finney_barycentric_insufficiency}
states that if $K,L$ are connected, locally finite simplicial complexes
such that $b(K)$ and $b(L)$ are isomorphic then $K$ is isomorphic
to $L$. Further, if $|K|$ is not a $1$-manifold and $K$ is neither
a simplex nor the boundary of a simplex then the isomorphism between $b(K)$ and $b(L)$ induces
an isomorphism of $K$ and $L$.

In other words, given $b(K)$ where $\dim K\ge2$ we can determine
the isomorphism type of $K$. Unless $K$ is the simplex or the boundary
of a simplex, we can further determine those subcomplexes of $b(K)$
that refine faces of $K$. The ``missing information'' in the exceptional
cases is the information assigning to each vertex $v$ of $b(K)$
the dimension of the minimal face $\sigma$ of $K$ that contains
it (equivalently, the dimension of the face $\sigma$ of which $v$
is the barycenter).

Fix $d,r\in\mathbb{N}$ and choose $r\cdot(d+1)$ nonisomorphic complexes
$\Sigma_{(1,0)},\ldots,\Sigma_{(r,d)}$ from $\mathcal{S}_{d}$. Let
a triangulated homology $d$-manifold $\Delta$ be given together
with a facet coloring $c:\Delta^{(d)}\rightarrow[r]$. We use the
following observation: for any face $\sigma\in\Delta$ there exists
a vertex $v_{\sigma}\in b(\Delta)$ at the barycenter of $\sigma$,
and the stars of the vertices $\left\{ v_{\sigma}\right\} _{\sigma\in\Delta}$
within $b(b(\Delta))$ partition the facets of $b(b(\Delta))$. Apply
\Cref{lem:coloring_d_plus_2} to $b(b(\Delta))$ (in which all vertex
degrees are at least $d+2$ by \Cref{lem:barycentric_d_plus_2}) as
follows: given a facet $\tau$ in the star of $v_{\sigma}$ within
$b(b(\Delta))$, $\tau$ is contained in a unique facet $\tau'$ of
$\Delta$. Color $\tau$ by $(c(\tau'),\dim\sigma)$, or in other
words, replace $\tau$ by $\Sigma_{(c(\tau'),\dim\sigma)}$. Perform
this process simultaneously for all facets of $b(b(\Delta))$, and
call the result $\mathrm{Enc}_{d,r}(\Delta,c)$. Note that the ``color
palette'' $\Sigma_{(1,0)},\ldots,\Sigma_{(r,d)}$ is part of $\mathrm{Enc}_{d,r}(\Delta,c)$,
and is also needed in order to compute the left inverse.

Statements (1,2) of the theorem follow automatically from the construction.
For statement (3), note that \Cref{lem:coloring_d_plus_2} guarantees
we can reconstruct the isomorphism type of $b(b(\Delta))$ together
with its facet coloring from the isomorphism type of $\mathrm{Enc}_{d,r}(\Delta,c)$.
We can then reconstruct $b(\Delta)$ together with, for each vertex,
the dimension of the face of $\Delta$ of which it is a barycenter;
this allows us to reconstruct the isomorphism type of $\Delta$. Further,
the given facet coloring of $b(b(\Delta))$ by colors in $(1,0),\ldots,(r,d)$
induces the coloring $c$ on the facets of $\Delta$ by taking just
the first coordinate, because all facets of $b(b(\Delta))$ refining
a given facet of $\Delta$ have colors with the same first coordinate.
\end{proof}

\section{Encoding manifold triangulations in colorings of the dual graph}

Given a triangulated $d$-manifold $M$, the dual graph $D$ does
not always uniquely determine the simplicial isomorphism type of $M$:
examples of this phenomenon exist in triangulations of the torus and
the projective plane (\cite{Ceballos_Doolittle}). We produce colorings
of $D$ with constantly many colors (depending on $d$) from which
$M$ can be reconstructed.

We begin by describing the data our colorings are designed to encode.
\begin{defn}
Let $<$ be a total order on $M^{(0)}$, and denote the vertices by
$v_{1}<\ldots<v_{n}$. Each facet $\sigma$ of $M$ is of the form
$\left\{ v_{i_{1}},\ldots,v_{i_{d+1}}\right\} $ (where the indices
are chosen such that $i_{1}<i_{2}<\ldots<i_{d+1}$). If $\tau$ is
a facet of $M$ that intersects $\sigma$ at a $(d-1)$-face then
$\sigma\setminus\tau=\left\{ v_{i_{j}}\right\} $ is a singleton,
and it has a well-defined index $j$ within the order induced on the
vertices of $\sigma$. In this situation, denote
\[
\widetilde{\mathrm{missing}}_{<}(\sigma,\tau)=j.
\]
Thus $\widetilde{\mathrm{missing}}_{<}$ is a function from ordered
pairs of adjacent facets into $\{1,\ldots,d+1\}$. By identifying
ordered pairs of adjacent facets in $M$ with the corresponding oriented
edges of $D$, we obtain a function
\[
\mathrm{missing}_{<}:\left\{ (v,w)\in D^{(0)} \times D^{(0)} \mid\left\{ v,w\right\} \text{ is an edge of \ensuremath{D}}\right\} \rightarrow\left\{ 1,\ldots,d+1\right\} .
\]
This is the \emph{missing vertex function} associated to $<$.
\end{defn}

\begin{claim}
Let $M$ be a triangulated homology $d$-manifold and let $<$ be
any linear order on $M^{(0)}$. The simplicial isomorphism type of
$M$ can be reconstructed from its dual graph $D$ together with the
function $\mathrm{missing}_{<}$.
\end{claim}

\begin{rem}
More generally, the proof works whenever all maximal faces in $M$
have dimension $d$, and the dual graph of the star of each vertex
of $M$ is connected.
\end{rem}

\begin{proof}
Let $X$ be the simplicial complex given by a disjoint union of $d$-simplices
in bijection with $D^{(0)}$ (together with all their faces). Denote
the vertices of a $d$-simplex $\sigma$ corresponding to $v\in D^{(0)}$
by $u_{v,1},\ldots,u_{v,d+1}$. Define an equivalence relation $\sim$
on the vertices of $X$: whenever $\{v,w\}$ is an edge of $D$ with
$i_{v}=\mathrm{missing}_{<}(v,w)$ and $i_{w}=\mathrm{missing}_{<}(w,v)$,
define corresponding vertices in the tuples $\left(u_{v,1},\ldots,\hat{u}_{v,i_{v}},\ldots,u_{v,d+1}\right)$
and $\left(u_{w,1},\ldots,\hat{u}_{w,i_{w}},\ldots,u_{w,d+1}\right)$
to be equivalent, where $\hat{u}_{v,i_{v}}$ denotes that the element
is excluded from the tuple. For a face $\tau\in X$, denote $\left[\tau\right]_{\sim}=\left\{ [u]_{\sim}\mid u\in\tau\right\} $
(the set of equivalence classes of vertices of $\tau$), and finally
set 
\[
X/\sim\coloneqq\left\{ [\tau]_{\sim}\mid\tau\in X\right\} .
\]
There is a simplicial map $\varphi:X\rightarrow M$ which is bijective
on $d$-faces: if $\left\{ u_{w,1},\ldots,u_{w,d+1}\right\} $ is
a facet of $X$ and $\sigma\in M^{(d)}$ is the facet corresponding
to $w\in D^{(0)}$, the vertices of $\sigma$ are $v_{i_{1}}<\ldots<v_{i_{d+1}}$
(in the order induced by $<$). Define $\varphi\left(u_{w,j}\right)=v_{i_{j}}$
for each $1\le j\le d+1$. Since each face of $M$ is contained in
a $d$-face, $\varphi$ is surjective on $i$-faces for all $i$.
It suffices to prove that $\varphi$ factors through the quotient
$X/\sim$ and that the induced simplicial map $X/\sim\rightarrow M$
is an isomorphism.

Observe that if $v\in M^{(0)}$ then $\varphi^{-1}(v)$ is a collection
of vertices, one contained in each facet of $X$ that corresponds
(via $\varphi$) to a facet of $\mathrm{st}_{v}(M)$. If $\sigma,\tau$
are two such facets of $X$ which are adjacent in the dual graph of
$\mathrm{st}_{v}(M)$, then $\sim$ identifies the unique vertex in
$\sigma\cap\varphi^{-1}(v)$ with the unique vertex in $\tau\cap\varphi^{-1}(v)$.
Since the dual graph of $\mathrm{st}_{v}(M)$ is connected, $\varphi^{-1}(v)$
is precisely an equivalence class under $\sim$, so $\varphi$ induces
a bijection $\left(X/\sim\right)^{(0)}\rightarrow M^{(0)}$. It also
follows that if $\rho\in X$ is any face then 
\[
\varphi\left(\left[\rho\right]_{\sim}\right)=\left\{ \varphi\left(\left[u\right]_{\sim}\right)\mid u\in\rho\right\} =\left\{ \varphi\left(u\right)\mid u\in\rho\right\} =\varphi\left(\rho\right),
\]
so $\varphi$ is well defined on the faces of $X/\sim$. We have $\varphi^{-1}\left(\varphi\left(\rho\right)\right)=\left\{ \left[u\right]_{\sim}\mid u\in\rho\right\} =\left[\rho\right]_{\sim}$,
and hence the induced map is injective on faces of $X/\sim$ (surjectivity
follows from the surjectivity of $\varphi$).
\end{proof}
\begin{prop}
\label[proposition]{prop:dual_graph_encoding}Let $D$ be the dual
graph of a triangulated homology $d$-manifold $M$. There is a vertex
coloring $c$ of $D$ with boundedly many colors (depending on $d$)
such that $\left(D,c\right)$ uniquely reconstructs the simplicial
isomorphism type of $M$.
\end{prop}

\begin{proof}
Choose a linear ordering $<$ of the vertices of $M$ and define the
function $\mathrm{missing}_{<}$ as above. We first pick a coloring
$c_{0}$ of $D$ such that any two vertices at distance at most $2$
have distinct colors. By Brooks' theorem this requires at most $k=\left(d+1\right)^{2}+1$
colors. We now define a coloring $c_{1}$ of $D$: for each $v\in D^{(0)}$,
list its neighbors $w_{1},\ldots,w_{d+1}$ (in some arbitrary ordering)
and define $c_{1}(v)$ to be the following tuple of pairs:
\[
c_{1}(v)=\left(c_{0}\left(w_{1}\right),\mathrm{missing}_{<}\left(v,w_{1}\right)\right),\ldots,\left(c_{0}\left(w_{d+1}\right),\mathrm{missing}_{<}\left(v,w_{d+1}\right)\right)\in\left(\left[k\right]\times\left[d+1\right]\right)^{d+1}.
\]
Now define $c(v)=\left(c_{0}(v),c_{1}(v)\right)$ for each $v\in D^{(0)}$.
This is a coloring with at most $k\cdot\left(k\left(d+1\right)\right)^{d+1}$
colors. It is clear that knowing $c$ allows us to reconstruct $\mathrm{missing}_{<}$,
and hence the simplicial isomorphism type of $M$.
\end{proof}

\section{Graphs and triangulated manifolds}

\label[section]{sec:graphs_and_manifolds}

We put together our previous results to prove the main results of
the paper.
\begin{thm}
\label[theorem]{thm:complexes_to_manifolds}Let $\mathcal{C}$ be
a class of simplicial $2$-complexes in which all vertex degrees are
bounded by some $k\in\mathbb{N}$. Fix a dimension $d\ge3$. Then
there is a function
\[
F:\mathcal{C}\rightarrow\left\{ \text{triangulated \ensuremath{d}-manifolds}\right\} 
\]
and a constant $c>0$ satisfying that for each $\Delta\in\mathcal{C}$:
\begin{enumerate}
\item If $\Delta$ has $n$ vertices then $F(\Delta)$ has at most $cn$
facets.
\item If $\Delta$ is nulhomologous over a field $\mathbb{F}$ then $F(\Delta)$
is a homology sphere over $\mathbb{F}$; if $\Delta$ is collapsible
then $F(\Delta)$ is a $d$-sphere; if $\Delta$ is contractible and
$d\ge4$ then $F(\Delta)$ is a homotopy sphere.
\item $F(\Delta)$ is quasi-isometric to the $1$-skeleton of $\Delta$
with quasi-isometry constant $c$.
\item $F$ is injective on isomorphism types: if $F(\Delta)$ is simplicially
isomorphic to $F(\Delta')$ then $\Delta$ is simplicially isomorphic
to $\Delta'$.
\end{enumerate}
\end{thm}

The idea is to thicken $\Delta$ into a triangulated manifold $M(\Delta)$
and pass to the boundary. The output of $F$ is more or less $\partial M(\Delta)$:
this satisfies all requirements except possibly the injectivity of
the construction, which requires us to encode additional information
in the triangulation. We do this using the results of \Cref{sec:encoding_facet_colorings}.
\begin{proof}
From the complexes $\Delta\in\mathcal{C}$ we construct $\left(d+1\right)$-dimensional
triangulated manifolds with boundary $M(\Delta)$ using \Cref{thm:handle_construction}
(the input degree bound is the degree bound of complexes in $\mathcal{C}$).
For such a manifold $M(\Delta)$, we color the facets of $\partial M(\Delta)$
with $3$ colors, as follows: if $\sigma$ is such a facet then is
contained in a unique handle $H$ of $M(\Delta)$, and $H$ is an
$i$-handle corresponding to an $i$-face of $\Delta$. Assign $\sigma$
the color $c(\sigma)=i$.
\begin{claim*}
The isomorphism type of $\Delta$ can be reconstructed from the colored
triangulation $\left(\partial M(\Delta),c\right)$.
\end{claim*}
\begin{proof}
Observe that for each handle $H\subset M(\Delta)$ the boundary $\partial H\subset M(\Delta)$
is a triangulated $d$-sphere. The submanifold-with-boundary $\partial H\cap\partial M(\Delta)$
has path-connected interior: its complement within $\partial H$ is
exactly the intersection of $H$ with the other handles of $M(\Delta)$,
and these intersections are regular neighborhoods of points, arcs,
and circles within $\partial H$. In particular, $\partial H\setminus\left(\partial H\cap\partial M(\Delta)\right)$
is the regular neighborhood of a graph. The complement of a graph
in a $d$-manifold (where $d\ge3$) is path connected, and hence so
is the complement of its closed regular neighborhood. In particular,
$\partial H\cap\partial M(\Delta)$ is a triangulated $d$-manifold
with boundary that has a connected dual graph. 

Note that for handles $H,H'\subset M(\Delta)$ corresponding to faces
$\rho,\rho'$ of $\Delta$ we have $H\cap H'\neq\emptyset$ if and
only if $\rho\subset\rho'$ or $\rho'\subset\rho$, and in this case
also $\partial H\cap\partial H'\cap\partial M$ intersect in a $\left(d-1\right)$-dimensional
subcomplex. Given $\left(\partial M(\Delta),c\right)$, let $D$ be
the dual graph of $\partial M(\Delta)$ (so that $c$ induces a vertex
coloring of $D$). The connected components in $D$ of each color
class $c(i)$ are in bijection with the $i$-faces of $\Delta$, and
a connected component $K$ of a color class $c(i)$ has an edge to
a connected component $L$ of a color class $c(i+1)$ if and only
if the face of $\Delta$ corresponding to $K$ is contained in the
face of $\Delta$ corresponding to $L$. Thus the isomorphism type
of $\Delta$ can be reconstructed $(D,c)$.
\end{proof}
Apply \Cref{thm:color_encoding} to the colored triangulations $\left(\partial M(\Delta),c\right)$
we obtain from $\mathcal{C}$. Composing these constructions gives
a function 
\[
F:\mathcal{C}\rightarrow\left\{ \text{triangulated \ensuremath{d}-manifolds}\right\} .
\]
This function is injective on isomorphism classes by the claim above
and by \Cref{thm:color_encoding}, which proves (4). The fact that
the total number of faces in $\Delta$ is at most a constant times
the number of vertices (because of the degree bound) together with
part (1) of \Cref{thm:color_encoding} and part (3) of \Cref{thm:handle_construction}
implies part (1). Part (2) follows from the fact that $\mathrm{Enc}_{d,3}\left(\partial M(\Delta),c\right)$
is PL-homeomorphic to $\partial M(\Delta)$, together with the fact
(part (2) of \Cref{thm:handle_construction}) that $M(\Delta)\searrow\Delta$:
this implies that if $\Delta$ is collapsible then so is $M(\Delta)$,
and in general $\Delta$ is homotopy equivalent to $M(\Delta)$, and
the topological conclusions follow directly from \Cref{thm:boundary}.
It remains to prove (3).

Let $\Delta\in\mathcal{C}$ and let $X=F(\Delta)=\mathrm{Enc}_{d,3}(\partial M(\Delta),c)$
be the corresponding triangulated manifold. Denote the dual graph
of $X$ by $D$. Since $X$ is a subdivision of $\partial M(\Delta)$
in which each facet is replaced by a combinatorial disk with (uniformly)
boundedly many facets, for each handle $H\subset M(\Delta)$ we have
a subdivision of $H\cap\partial M(\Delta)$ induced by $X$. Denote
the set of vertices of $D$ corresponding to the facets of $H\cap\partial M(\Delta)$
by $D_{H}$. Then $\left\{ D_{H}\mid H\text{ is a handle of \ensuremath{M}}\right\} $
partitions the vertex set of $D$. By construction, each $D_{H}$
contains boundedly many vertices, with the bound depending only on
the degree bound of complexes in $\mathcal{C}$ and on the dimension
$d$. By the proof of the claim above, each $D_{H}$ is also nonempty
and induces a connected subgraph of $D$. Fix $M$ to be a uniform
bound on the number of vertices in a subset $D_{H}$, and note that
this is also a bound on the diameter of the induced subgraph of such
a set. Define a relation
\[
R\subset\left(\text{vertices of \ensuremath{D}}\right)\times\left(\text{vertices of \ensuremath{\Delta}}\right)
\]
so that $xRv$ if and only if $x\in D_{H}$ where $H$ is a handle
of $M$ corresponding to a face $\sigma\in\Delta$ satisfying $v\in\sigma$.
We prove that $R$ satisfies the assumptions of \Cref{lem:quasi_isometric_relations}:
\begin{itemize}
\item If $x$ is a vertex of $D$ then $x\in D_{H}$ for some $H$, and
the face of $\Delta$ corresponding to $H$ contains a vertex $v$
for which $xRv$.
\item If $v$ is a vertex of $\Delta$ and $H$ is the corresponding handle
of $M(\Delta)$ then $D_{H}\neq\emptyset$, so there is some $x\in D_{H}$
with $xRv$.
\item Let $x,x'$ be vertices of $D$ at distance at most $1$, and assume
$xRv$ and $x'Rv'$. Then $x\in D_{H}$ and $x'\in D_{H'}$ for $H,H'$
handles of $M$ corresponding to either the same face $\sigma'=\sigma\in\Delta$
or (without loss of generality) to a pair $\sigma'\subset\sigma$
of faces of $\Delta$. Further, $v\in\sigma$ and $v'\in\sigma'$
by definition of $R$. Therefore either $v=v'$ or $v,v'$ are adjacent
in the $1$-skeleton (since both are vertices of the same face $\sigma$).
Hence $d(v,v')\le1$.
\item Let $v,v'$ be vertices of $\Delta$ at distance at most $1$, and
assume $xRv$ and $x'Rv'$. Then $x\in D_{H}$ and $x'\in D_{H'}$
for $H,H'$ handles of $M$ corresponding to faces $\sigma,\sigma'$
containing $v$ and $v'$ respectively. If $H_{v}$ and $H_{v'}$
are the handles corresponding to $v$ and $v'$, then there is an
edge between each consecutive pair of subsets in the sequence $D_{H},D_{H_{v}},D_{H_{v'}},D_{H'}$,
and these subsets induce connected subgraphs of diameter at most $M$.
Hence $d(x,x')\le4M+3$.
\end{itemize}
By \Cref{lem:quasi_isometric_relations}, $D$ is quasi-isometric
to $\Delta^{\le1}$ with constants depending only on the dimension
$d$ and the degree bound $k$ of complexes in $\mathcal{C}$.
\end{proof}
\begin{cor}
[\Cref{thm:expanding_S3s}]For each $d \ge 3$ there is an infinite family of combinatorially
distinct triangulated $d$-spheres such that their dual graphs form
a family of $(d+1)$-regular expanders. Moreover, there is a uniform bound (depending only on the dimension $d$) on the number of simplices incident to each vertex of these triangulations.
\end{cor}

\begin{proof}
It suffices to construct an infinite family of $2$-complexes $\left\{ \Delta_{n}\right\} _{n\in\mathbb{N}}$
such that each $\Delta_{n}$ is collapsible, the vertex degrees in
the family are uniformly bounded by some $k$, and the $1$-skeleta
$\left\{ \Delta_{n}^{\le1}\right\} _{n\in\mathbb{N}}$ form an expander
family: applying \Cref{thm:complexes_to_manifolds}
to such a family gives a corresponding family of triangulated $d$-spheres
$\left\{ \Sigma_{n}\right\} _{n\in\mathbb{N}}$ such that the dual
graph of each $\Sigma_{n}$ is quasi-isometric to $\Delta_{n}^{\le1}$,
with uniformly bounded quasi-isometry constants. By \Cref{fact:expanders_via_quasi_isometry},
this implies that the dual graphs to $\left\{ \Sigma_{n}\right\} _{n\in\mathbb{N}}$
are an expander family. We construct a suitable family $\left\{ \Delta_{n}\right\} _{n\in\mathbb{N}}$
in \Cref{thm:collapsible_expanding_complexes}.
\end{proof}

\begin{cor}[\Cref{cor:expanding_manifolds}]
    Let $M$ be a triangulable $d$-manifold. Then there is an infinite family of combinatorially distinct triangulations of $M$ in which the vertex degrees are uniformly bounded (with the bound depending on $M$) and for which the dual graphs form a family of $(d+1)$-regular expanders. If $M$ is PL, there exists such a family in which the triangulations are combinatorial.
\end{cor}

\begin{proof}
    Fix a triangulation $X$ of $M$, taking $X$ combinatorial if $M$ is PL. Let $\{\Sigma_n\}_{n \in \mathbb{N}}$ be a family of triangulated $d$-spheres as in \Cref{thm:expanding_S3s} (see the previous corollary). Fix a facet $\sigma_0$ of $X$, and for each $n$ choose a facet $\sigma_n$ of $\Sigma_n$, together with a simplicial isomorphism $f_n: \partial \sigma_0 \to \partial \sigma_n$. Delete the interiors of $\sigma_0$ and of $\sigma_n$ and glue $X$ to $\Sigma_n$ along the boundaries of the two facets via $f_n$: this yields a simplicial complex $Y_n$ which is PL homeomorphic to $X \# S^d \simeq X$. (Note that the gluing is indeed simplicial: $\partial \sigma_0$ is an induced subcomplex of $X \setminus \sigma_0$, and similarly for $\partial \sigma_n$ and $\Sigma_n \setminus \sigma_n$.)

    We prove that the dual graphs of $\{Y_n\}_{n \in \mathbb{N}}$ are an expander family. By \Cref{fact:expanders_via_quasi_isometry}, it suffices to show that they are quasi-isometric (with uniform constants) to the dual graphs of $\{\Sigma_n\}_{n \in \mathbb{N}}$. For each $n\in \mathbb{N}$ construct the following relation $R_n \subset Y_n^{(d)} \times \Sigma_n^{(d)}$: 
    \[R_n = \{(\rho,\sigma_0) \mid \rho \in X^{(d)}\setminus \{\sigma_0\}\} \cup \{(\rho, \rho) \mid \rho \in \Sigma_n^{(d)} \setminus \{\sigma_n\}\}.\]
    That is, a pair of facets $(\rho,\sigma)$ is in $R_n$ if $\rho=\sigma$ is a facet of $\Sigma_n \setminus \sigma_n$ or if $\rho$ is a facet of $X \setminus \sigma_0$ and $\sigma=\sigma_n$. Identifying the facets of $X$ with the vertices of its dual graph $D(X)$ and similarly for $\Sigma_n$ and $Y_n$, each such relation satisfies the conditions of \Cref{lem:quasi_isometric_relations} with $L=c=1$ and $M=\mathrm{diam}(D(X))+1$, so we obtain the desired quasi-isometries.
\end{proof}

\begin{lem}
\label[lemma]{lem:uniform_shortness_of_output}Let $\mathcal{C}$
be a class of bounded-degree simplicial $2$-complexes which are nulhomologous
over $\mathbb{F}$. Then the corresponding family of triangulated
$3$-dimensional homology manifolds output by \Cref{thm:complexes_to_manifolds}
has dual graphs which are uniformly short over $\mathbb{F}$. The
bound $L$ on the length of a cycle in a homology basis for these
graphs depends only on the degree bound for complexes in $\mathcal{C}$,
and not on $\mathbb{F}$.
\end{lem}

\begin{proof}
A homology basis for the dual graph of a triangulated homology sphere
$\Sigma$ is given by the dual $2$-cells (in the case of a $3$-manifold,
these are the cycles in the dual graph corresponding to the links
of $1$-simplices in $\Sigma$). If $\Sigma$ is the output of running
\Cref{thm:complexes_to_manifolds} on a $2$-complex with all vertex
degrees bounded by $k$ then the star of any face $\sigma\in\Sigma$
is contained in the union of a uniformly bounded number $m$ of handles
(depending only on $k$) and each such handle contains a uniformly
bounded number of faces $N$ (again depending only on $k$). Hence
the star of $\sigma$ contains at most $mN$ faces. Note that the
construction of $\Sigma$ does not depend on $\mathbb{F}$, so the
bound $L=mN$ in the statement depends only on $k$.
\end{proof}
\begin{lem}
\label[lemma]{lem:4_to_3}Let $\mathcal{C}$ be a class of $4$-regular
short graphs with parameters $k,L$. Then there is a class $\mathcal{D}$
of $3$-regular short graphs with parameters $k',L'$ depending only
on $k,L$ and a map $\mathcal{C}\rightarrow\mathcal{D}$
that is injective on isomorphism classes and increases the number of vertices by only a constant factor.
\end{lem}

\begin{proof}
Given a $4$-regular graph, replace each vertex (with its four half-edges)
with a cycle on $4$-vertices with the four half-edges shown:

\includegraphics{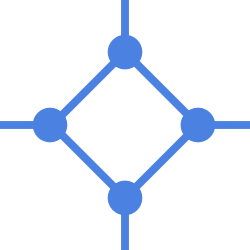}

then connect each half-edge to a half-edge of a $4$-cycle of a neighboring
vertex. This construction is injective: contracting all $4$-cycles
in the resulting graph to vertices always yields the original graph
(any cycle in the original graph is subdivided into a cycle of length
at least $6$). Further, the resulting graphs are short: each cycle
of the input graph of given length $m$ is subdivided into a cycle
of length at most $3m$, and together with the $4$-cycles these form
a homology basis. It is straightforward that each edge of the new
graph participates in a bounded number $k'$ of cycles in this basis,
depending on $k$.
\end{proof}
\begin{lem}
\label[lemma]{lem:uniformly_short_encoding}There is an integer $L_{0}\in\mathbb{N}$
such that the following holds. Let $\mathbb{F}$ be a field and let
$\mathcal{D}$ be the class of uniformly short graphs of degree $3$
over $\mathbb{F}$ in which there is a cycle basis consisting of cycles
of length $\le L_{0}$. For any class $\mathcal{C}$ of short graphs
over $\mathbb{F}$ there is a constant $c$ and a function $S:\mathcal{C}\rightarrow\mathcal{D}$, injective on isomorphism classes,
such that every $G\in\mathcal{C}$ with $n$
vertices is mapped to a graph $S(G)$ with at most $cn$ vertices.
\end{lem}

The point is that while $c$ depends on $\mathcal{C}$, $L_{0}$ doesn't.
\begin{proof}
Observe that if $\Sigma$ is a triangulated $3$-manifold then the
dual graph of $b(b(\Sigma))$ uniquely determines the simplicial isomorphism
type of $\Sigma$ (alternatively one can avoid using this observation
by an appropriate use of \Cref{prop:dual_graph_encoding}). We perform
the following construction twice: given a short class $\mathcal{D}$
over $\mathbb{F}$, first use \Cref{prop:graphs_to_2_complexes} to
fill the cycles and obtain a class of $2$-complexes. Then apply \Cref{thm:complexes_to_manifolds}
to obtain triangulated $3$-manifolds, and finally take the second
barycentric subdivision of each of these and return the dual graphs.
This construction is injective on $\mathcal{D}$ (because each step
separately is injective) and it increases the number of vertices in
the input graphs by only a constant factor, which depends on $\mathcal{D}$.
The result is a short class $\mathcal{D}'$ of $4$-regular graphs. 

Apply this first to $\mathcal{D}=\mathcal{C}$, and then again to
the resulting class $\mathcal{D}'$, to obtain some short class $\mathcal{D}''$.
By \Cref{lem:uniform_shortness_of_output}, the output class $\mathcal{D}''$
is uniformly short, with cycle length bound $L''$ depending on the
maximal vertex degrees in the simplicial $2$-complexes that result
from filling in the graphs of $\mathcal{D}'$. By the proof of \Cref{prop:graphs_to_2_complexes},
these vertex degrees depend only on the class $\mathcal{D}'$: specifically,
if each $G\in\mathcal{D}'$ has a homology basis $\left\{ \gamma_{i}\right\} _{i}$
such that each edge of $G$ participates in at most $k$ of the cycles,
and the vertex degrees in each $G\in\mathcal{D}$ are bounded by $d$,
then the vertex degrees of the $2$-complexes are at most $d(1+4k)$.
Since $\mathcal{D}'$ consists of dual graphs of triangulated $3$-spheres,
we can take $d=4$ and $k=3$ (each edge of the dual graph participates
in precisely three dual $2$-faces, see \Cref{prop:sphere_duals_short}).
Hence $L''$ is a universal constant.

Finally, the class $\mathcal{D}''$ is $4$-regular. Applying \Cref{lem:4_to_3}
injects it into a uniformly short $3$-regular class (with constant
cycle bound the desired $L_{0}$) and finishes the proof.
\end{proof}
\begin{cor}
[\Cref{thm:triangulations_vs_short_graphs}]There exists $L_{0}\in\mathbb{N}$
such that for all $L\ge L_{0}$ the following holds. Let $\mathbb{F}$
be a field and $d\ge3$ an integer. Let $s_{d,\mathbb{F}}\left(N\right)$
be the number of isomorphism types of triangulated combinatorial $d$-manifolds
$M$ with $H_{1}\left(M;\mathbb{F}\right)=0$ and with at most $N$
$d$-simplices. Let $t_{L,\mathbb{F}}\left(N\right)$ be the number
of connected $3$-regular graphs $G$ with at most $N$ vertices such
that cycles of length $\le L$ generate $H_{1}\left(G;\mathbb{F}\right)$.
Then there exists $c,c'>0$ such that
\[
s_{d,\mathbb{F}}(\frac{1}{c}N)\le t_{L,\mathbb{F}}(N)\le s_{d,\mathbb{F}}(cN)
\]
for all $N\in\mathbb{N}$.
\end{cor}

\begin{proof}
The inequality $t_{L,\mathbb{F}}(N)\le s_{d,\mathbb{F}}(cN)$ follows
from \Cref{prop:graphs_to_2_complexes} and \Cref{thm:complexes_to_manifolds}:
the uniformly short class $\mathcal{C}$ consisting of $3$-regular
graphs with $\mathbb{F}$-homology bases having all cycles of length
at most $L$ can be injectively encoded by combinatorially triangulated
$\mathbb{F}$-homology $d$-spheres in such a way that the encoding
of a graph $G\in\mathcal{C}$ on $n$ vertices has at most $cn$ facets
for some fixed $c>0$.

For the inequality $s_{d,\mathbb{F}}(\frac{1}{c}N)\le t_{L,\mathbb{F}}(N)$
we encode triangulated $d$-dimensional $\mathbb{F}$-homology spheres
in $3$-regular short graphs as follows. Given a triangulated homology
sphere $\Sigma$, first pass to the dual graph, colored using \Cref{prop:dual_graph_encoding}
(so it can be used to reconstruct $\Sigma$). This is a $\left(d+1\right)$-regular
colored graph $\left(D,c\right)$, where $c$ uses $k$ colors, with
$k$ depending only on $d$. Encode it into a bounded-degree graph
$D^{\prime}$ as follows: if $c(v)=t\in\left[k\right]$, add $t$
isolated vertices to the graph and connect them to $v$. This construction
yields bounded degree graphs (with degree at most $d+1+k$,) increases
the number of vertices by a constant depending only on $d$, and is
injective: $D$ can be reconstructed from $D^{\prime}$ by erasing
all degree-$1$ vertices, and the color $c(v)$ of a given vertex
is the number of degree-$1$ vertices connected to $v$. Now by \Cref{prop:sphere_duals_short}
and \Cref{lem:uniformly_short_encoding} we can encode these graphs
injectively into $3$-regular graphs with first homology generated
by cycles of length at most $L_{0}$.
\end{proof}

\section{Sphere triangulations with expanding duals}

\label[section]{sec:telescopes} We construct an infinite sequence
of bounded-degree collapsible $2$-complexes such that their $1$-skeleta
are an expander family. The idea is to construct the mapping telescopes
of sequences of $2$-covers obtained from \cite{Bilu_Linial}.

\subsection{Collapsibility and expansion of telescopes}
\begin{defn}
Let $\left\{ X_{k}\right\} _{k=1}^{n}$ be a sequence of cell complexes,
and $\left\{ f_{k}\right\} _{k=1}^{n-1}$ a sequence of maps $f_{k}:X_{k+1}\rightarrow X_{k}$.
The \emph{mapping telescope} of the sequence is 
\[
\frac{\left(X_{1}\times\{1\}\right)\sqcup\bigsqcup_{k=2}^{n}X_{k}\times I}{\sim},
\]
where the relation $\sim$ identifies $\left(x,0\right)$ with $\left(f_{k}(x),1\right)$
whenever $x\in X_{k+1}$.
\end{defn}

In other words, the mapping telescope is given by gluing the mapping
cylinders of the maps $\left\{ f_{k}\right\} _{k}$ together, each
on top of the previous one.
\begin{defn}
A \emph{hierarchical sequence of graphs} $\mathcal{H}$ is a sequence
of graphs $\left\{ H_{n}\right\} _{n\in\mathbb{N}}$ together with
functions $f_{n}:V(H_{n+1})\rightarrow V(H_{n})$ for each $n\in\mathbb{N}$,
satisfying that for each pair $\{v,w\}\in E(H_{n+1})$, we have either
$f_{n}(v)=f_{n}(w)$ or $\left\{ f_{n}(v),f_{n}(w)\right\} \in E\left(H_{n}\right)$. 

Given a hierarchical sequence $\mathcal{H}=\left(\left\{ H_{n}\right\} _{n\in\mathbb{N}},\left\{ f_{n}\right\} _{n\in\mathbb{N}}\right)$,
we extend each $f_{n}$ to a function $H_{n+1}\rightarrow H_{n}$
by mapping each edge $\left\{ v,w\right\} $ of $H_{n+1}$ linearly
onto $\left\{ f(v),f(w)\right\} $ (which may be a single vertex).
The $N$-th mapping telescope $T_{N}\left(\mathcal{H}\right)$ of
the sequence is then the mapping telescope of $\left\{ f_{n}\right\} _{n=1}^{N}$.
\end{defn}

Observe that $T_{N}\left(\mathcal{H}\right)$ is a $2$-dimensional
CW complex in which all $2$-cells are squares or triangles.
\begin{claim}
\label[claim]{claim:telescope_collapse}Let $\mathcal{H}=\left(\left\{ H_{n}\right\} _{n\in\mathbb{N}}\left\{ f_{n}\right\} _{n\in\mathbb{N}}\right)$
be a hierarchical sequence of graphs. Then $T_{N}\left(\mathcal{H}\right)$
collapses onto $H_{1}$ for all $N\in\mathbb{N}$. In particular,
if $H_{1}$ is a single vertex, then each $T_{N}\left(\mathcal{H}\right)$
is collapsible.
\end{claim}

\begin{proof}
It suffices to show that $T_{1}(\mathcal{H})$ collapses onto $H_{1}$, and that
$T_{N}\left(\mathcal{H}\right)$ collapses onto $T_{N-1}\left(\mathcal{H}\right)$
for each $N\ge2$. If $N\ge2$, note that
\[
T_{N}\left(\mathcal{H}\right)\simeq\left[T_{N-1}\left(\mathcal{H}\right)\sqcup\left(H_{N+1}\times I\right)\right]/\sim,
\]
where $\sim$ glues $H_{N+1}\times\{0\}$ onto $T_{N-1}\left(\mathcal{H}\right)$.
The $2$-cells of $T_{N}\left(\mathcal{H}\right)$ which are not contained
in $T_{N-1}\left(\mathcal{H}\right)$ have boundary $\left\{ v,w,f_{N}(v),f_{N}(w)\right\} $,
and in each of these $\left\{ v,w\right\} \in E\left(H_{N+1}\right)$
is a free face. Hence all such $2$-cells can be collapsed. We are
then left with the subcomplex
\[
\left[T_{N-1}\left(\mathcal{H}\right)\sqcup\left(V\left(H_{N+1}\right)\times I\right)\right]/\sim,
\]
and for each $1$-cell of the form $v\times I$ (where $v\in V\left(H_{N+1}\right)$)
the vertex $v\times\{1\}$ is a free face. Hence these can be collapsed
as well, leaving exactly $T_{N-1}\left(\mathcal{H}\right)$. 

In exactly the same way, $T_{1}\left(\mathcal{H}\right)=\left[\left(H_{1}\times\left\{ 1\right\} \right)\sqcup\left(H_{2}\times I\right)\right]/\sim$
collapses onto $H_{1}\times\left\{ 1\right\} $ by first collapsing
all $2$-cells, and then all $1$-cells of the form $v\times I$ for
$v\in V\left(H_{2}\right)$ (from the free face $v\times\{1\}$).
\end{proof}
\begin{prop}
\label[proposition]{prop:telescope_expansion} Let $\mathcal{H}=\left(\left\{ H_{n}\right\} _{n\in\mathbb{N}},\left\{ f_{n}\right\} _{n\in\mathbb{N}}\right)$
be a hierarchical sequence of graphs. Suppose $\left\{ H_{n}\right\} _{n\in\mathbb{N}}$
forms a bounded degree expander family, and that there exists a constant
$c$ such that $2\le\left|f_{n}^{-1}(v)\cap V\left(H_{n+1}\right)\right|\le c$
for all $n\in\mathbb{N}$ and $v\in V\left(H_{n}\right)$. Then the
sequence of graphs $\left\{ T_{N}\left(\mathcal{H}\right)^{\le1}\right\} _{N\in\mathbb{N}}$
is a bounded degree expander family.
\end{prop}

\begin{proof}
For convenience we denote $T_{N}\coloneqq T_{N}\left(\mathcal{H}\right)^{\le1}$
for all $N$. Note that if the degrees of the graphs $\left\{ H_{n}\right\} _{n\in\mathbb{N}}$
are bounded by some $d$ then the degrees of the graphs $\left\{ T_{N}\right\} _{N\in\mathbb{N}}$
are bounded by $d+c+1$. 

Suppose the graphs $\left\{ H_{n}\right\} _{n\in\mathbb{N}}$ all
have edge expansion ratio at least $\alpha$. Let $N\in\mathbb{N}$
and let $W\subset V\left(T_{N}\right)$ have $|W|\le\frac{1}{2}\left|V\left(T_{N}\right)\right|$. 

Recall $V\left(T_{N}\right)=\bigsqcup_{n\le N+1}V\left(H_{n}\right)$.
Call $V\left(H_{n}\right)$ the $n$-th \emph{layer} of $V\left(T_{N}\right)$
for each $n\le N+1$, and denote the part of $W$ in $n$-th layer
by $L_{n}=W\cap V\left(H_{n}\right)$. Let
\[
S=\bigcup\left\{ L_{n}\mid\frac{\left|L_{n}\right|}{\left|V\left(H_{n}\right)\right|}<\frac{9}{10}\right\} ,
\]
\[
D=\bigcup\left\{ L_{n}\mid\frac{\left|L_{n}\right|}{\left|V\left(H_{n}\right)\right|}\ge\frac{9}{10}\right\} .
\]
Intuitively, $S$ is the union of those layers $L_{n}$ of $W$ which
are reasonably sparse within $V\left(H_{n}\right)$, while $D$ is
the union of those layers which are dense. It is clear that $\left\{ S,D\right\} $
is a partition of $W$. We split into several cases, and show $\left|E\left(W,W^{c}\right)\right|$
is at least a constant multiple of $\left|W\right|$ in each case:

\textbf{Case 1:} $\left|S\right|\ge\frac{1}{2}\left|W\right|$. Let
$n$ be such that $L_{n}\subset S$, and observe that 
\[
\frac{1}{9}\left|L_{n}\right|\le\min\left\{ \left|L_{n}\right|,\left|V\left(H_{n}\right)\setminus L_{n}\right|\right\} \le\frac{1}{2}\left|V\left(H_{n}\right)\right|.
\]
Therefore 
\[
\left|E\left(L_{n},V\left(H_{n}\right)\setminus L_{n}\right)\right|\ge\frac{1}{9}\alpha\left|L_{n}\right|.
\]
Since $\bigcup_{n\le N+1}E\left(L_{n},V\left(H_{n}\right)\setminus L_{n}\right)\subset E\left(W,W^{c}\right)$,
we have
\[
E\left(W,W^{c}\right)\ge\sum_{\left\{ n\mid L_{n}\subset S\right\} }\frac{1}{9}\alpha\left|L_{n}\right|=\frac{1}{9}\alpha\cdot\left|S\right|\ge\frac{\alpha}{18}\left|W\right|.
\]

\textbf{Case 2:} $\left|D\right|\ge\frac{1}{2}\left|W\right|$. Let
$k=\max\left\{ n\le N+1\mid L_{n}\subset D\right\} $. We split this
into two further subcases.

\textbf{Case 2.1:} $k<N+1$. Then $L_{k+1}\subset S$. We have
\[
\left|E\left(W,W^{c}\right)\right|\overset{\star}{\ge}\left|E\left(L_{k},V\left(H_{k+1}\right)\setminus L_{k+1}\right)\right|+\left|E\left(L_{k+1},V\left(H_{k+1}\right)\setminus L_{k+1}\right)\right|.
\]

Note that $\left|L_{k}\right|\ge\frac{1}{4}\left|W\right|$, because
\[
\left|L_{k}\right|\ge\frac{9}{10}\left|V\left(H_{k}\right)\right|\ge\frac{9}{10}\cdot\frac{1}{2}\left|\bigcup_{j\le k}V\left(H_{j}\right)\right|,
\]
where $D\subseteq\bigcup_{j\le k}V\left(H_{j}\right)$ contains at
least half the vertices of $\left|W\right|$. 

The idea is now that either $L_{k+1}$ is relatively large compared
to $L_{k}$, in which case an analysis similar to case 1 applies to
the summand $\left|E\left(L_{k+1},V\left(H_{k+1}\right)\setminus L_{k+1}\right)\right|$,
or that $L_{k+1}$ is small compared to $L_{k}$, in which case $\left|E\left(L_{k},V\left(H_{k+1}\right)\setminus L_{k+1}\right)\right|$
is at least a constant times $\left|L_{k}\right|$.

In detail, if $\left|L_{k+1}\right|>\left|L_{k}\right|$ then (as
in case 1) we have 
\[
\left|E\left(W,W^{c}\right)\right|\overset{\text{by \ensuremath{\star}}}{\ge}\left|E\left(L_{k+1},V\left(H_{k+1}\right)\setminus L_{k+1}\right)\right|\ge\frac{1}{9}\alpha\left|L_{k+1}\right|\ge\frac{1}{9}\alpha\cdot\left|L_{k}\right|\ge\frac{\alpha}{36}\left|W\right|
\]
On the other hand, if $\left|L_{k+1}\right|\le\left|L_{k}\right|$,
then since each vertex of $L_{k}$ is connected to at least two vertices
of $V\left(H_{k+1}\right)\supseteq L_{k+1}$ (and each vertex of $H_{k+1}$
is connected to at most one vertex of $L_{k}$,) we find
\[
\left|E\left(W,W^{c}\right)\right|\overset{\text{by \ensuremath{\star}}}{\ge}\left|E\left(L_{k},V\left(H_{k+1}\right)\setminus L_{k+1}\right)\right|\ge\left|E\left(L_{k},V\left(H_{k+1}\right)\right)\right|-\left|E\left(L_{k},L_{k+1}\right)\right|
\]
\[
\ge2\left|L_{k}\right|-\left|L_{k}\right|\ge\left|L_{k}\right|\ge\frac{1}{4}\left|W\right|.
\]

\textbf{Case 2.2: }$k=N+1$. This case is very similar to case 2.1.
First note that $L_{N}\not\subset D$, because otherwise 
\[
\left|W\right|\ge\frac{9}{10}\left|V\left(H_{N}\right)\cup V\left(H_{N+1}\right)\right|\ge\frac{9}{10}\cdot\frac{3}{4}\left|V\left(T_{N}\right)\right| > \frac{1}{2}\left|V\left(T_{N}\right)\right|.
\]
(This follows from the fact that $\left|V\left(H_{i+1}\right)\right|\ge2\left|V\left(H_{i}\right)\right|$
for all $i\le N$.) Therefore 
\[
\left|E\left(W,W^{c}\right)\right|\ge\left|E\left(L_{N},V\left(H_{N}\right)\setminus L_{N}\right)\right|+\left|E\left(L_{N+1},V\left(H_{N}\right)\setminus L_{N}\right)\right|.
\]

Observe that 
\[
\left|E\left(L_{N+1},V\left(H_{N}\right)\setminus L_{N}\right)\right|=\left|E\left(L_{N+1},V\left(H_{N}\right)\right)\right|-\left|E\left(L_{N+1},L_{N}\right)\right|
\]
\[
=\left|L_{N+1}\right|-\left|f^{-1}\left(L_{N}\right)\cap L_{N+1}\right|\ge\left|L_{N+1}\right|-c\left|L_{N}\right|.
\]
Hence if $\left|L_{N}\right|\le\frac{1}{2c}\left|L_{N+1}\right|$
we have $\left|E\left(W,W^{c}\right)\right|\ge\frac{1}{2}\left|L_{N+1}\right|$.

On the other hand, if $\left|L_{N}\right|>\frac{1}{2c}\left|L_{N+1}\right|$,
we have $\left|E\left(L_{N},V\left(H_{N}\right)\setminus L_{N}\right)\right|\ge\frac{1}{9}\alpha\left|L_{N}\right|$
by the considerations of case 1, and hence $\left|E\left(W,W^{c}\right)\right|\ge\frac{1}{9}\alpha\left|L_{N}\right|\ge\frac{\alpha}{18c}\left|L_{N+1}\right|$.

In either case, since
\[
\left|L_{N+1}\right|\ge\frac{9}{10}\left|V\left(H_{N+1}\right)\right|\ge\frac{9}{20}\left|V\left(T_{N}\right)\right|\ge\frac{9}{20}\left|W\right|,
\]
we have a lower bound on $\left|E\left(W,W^{c}\right)\right|$ of
the form $\text{constant}\cdot\left|W\right|$.
\end{proof}

\subsection{An expanding family of collapsible 2-complexes}
\begin{thm}
\label[theorem]{thm:collapsible_expanding_complexes}There exists
an infinite family $\left\{ X_{n}\right\} _{n\in\mathbb{N}}$ of finite
$2$-dimensional simplicial complexes satisfying:
\begin{enumerate}
\item Each of the complexes is collapsible to a point,
\item The $1$-skeleta $\left\{ X_{n}^{\le1}\right\} _{n}$ are a bounded-degree
expander family.
\end{enumerate}
\end{thm}

\begin{proof}
By the main result of \cite{Bilu_Linial}, if $d\in\mathbb{N}$ is
large enough then there exists an expander family $\left\{ H_{n}\right\} _{n\in\mathbb{N}}$
of $d$-regular graphs such that each $H_{n+1}$ is a $2$-sheeted
cover of $H_{n}$. Denote the covering maps by $f_{n}:H_{n+1}\rightarrow H_{n}$.
By prefixing this sequence with a graph $H_{1}$ on a single vertex
(incrementing all other indices for consistency, and letting $f_{1}:H_{2}\rightarrow H_{1}$
be the constant map) we obtain a hierarchical family $\mathcal{H}$
of bounded degree expanders. The sequence of mapping telescopes $\left\{ T_{N}\left(\mathcal{H}\right)\right\} _{N\in\mathbb{N}}$
is the required family of $2$-complexes by \Cref{claim:telescope_collapse}
and \Cref{prop:telescope_expansion}.
\end{proof}

\bibliographystyle{plain}
\bibliography{spheres}

\end{document}